\documentclass[12pt]{article}
\usepackage{amsmath,amssymb,amsfonts,amsthm,url,xspace,graphicx,enumerate,xcolor,hyperref}
\newtheorem{theorem}{Theorem}
\numberwithin{theorem}{section}

\newtheorem{lemma}[theorem]{Lemma}
\newtheorem{prop}[theorem]{Proposition}

\theoremstyle{remark}

\newcommand{\Z}{\mathbb{Z}}
\newcommand{\R}{\mathbb{R}}
\newcommand{\C}{\mathbb{C}}
\newcommand{\G}{\mathcal{G}}
\renewcommand{\H}{\mathcal{H}}

\newcommand{\T}{\mathbb{T}}

\newcommand{\eps}{\varepsilon}
\newcommand{\be}{\begin{equation}}
\newcommand{\ee}{\end{equation}}
\newcommand{\old}[1]{}

\newcommand{\ST}{\nu}
\newcommand{\ent}{{\mathcal S}}

\newcommand{\om}{|\omega|}

\begin{document}
\title{Determinantal spanning forests on planar graphs}
\author{Richard Kenyon\thanks{Research supported by NSF Grants DMS-1208191, DMS-1612668 and Simons Investigator grant 327929}\\ Brown University\\ Providence, RI 02912, USA }
\date{}
\maketitle
\begin{abstract}
We generalize the uniform spanning tree 
to construct a family of determinantal measures on essential spanning forests on periodic planar graphs 
in which every component tree is bi-infinite.
Like the uniform spanning tree, these measures arise naturally from the laplacian on the graph. 

More generally these results hold for the ``massive" laplacian determinant which counts rooted spanning forests with weight $M$ per finite component. These measures typically have a form of conformal invariance, unlike the usual rooted spanning tree measure.
We show that the spectral curve for these models is always a simple Harnack curve; this fact controls the decay of edge-edge correlations
in these models.

We compute a limit shape theory in these settings, where the limit shapes are defined by measured foliations
of fixed isotopy type. 
\end{abstract}

\tableofcontents

\section{Introduction}

The relation between spanning trees and the laplacian on a graph was first discovered by Kirchhoff more than 
150 years ago \cite{Kirch}.
In the past 30 years
this relation has played an essential role in the development of a large part of probability theory and statistical mechanics \cite{Temperley, Aldous, Broder, 
Pemantle, BP, Wilson, Kenyon.confinv, Schramm, LSW, BLPS, Kenyon.bundle}. We define here a very natural generalization,
for periodic planar graphs, of Kirchhoff's results and of the uniform spanning tree
measure, to a $2$-parameter family of measures on spanning forests. These measures enjoy most of the properties of the uniform spanning tree, being determinantal, and in fact also arise from the laplacian determinant.
Taken together as a family we find additional behavior such as phase transitions and limit shapes.
\medskip

Figure \ref{esf1} shows part of a uniform random spanning tree of an infinite square grid graph on a strip of width $4$.
Such measures were constructed by Pemantle \cite{Pemantle} as limits of measures on finite graphs.

\begin{figure}[htbp]
\begin{center}\includegraphics[width=3in]{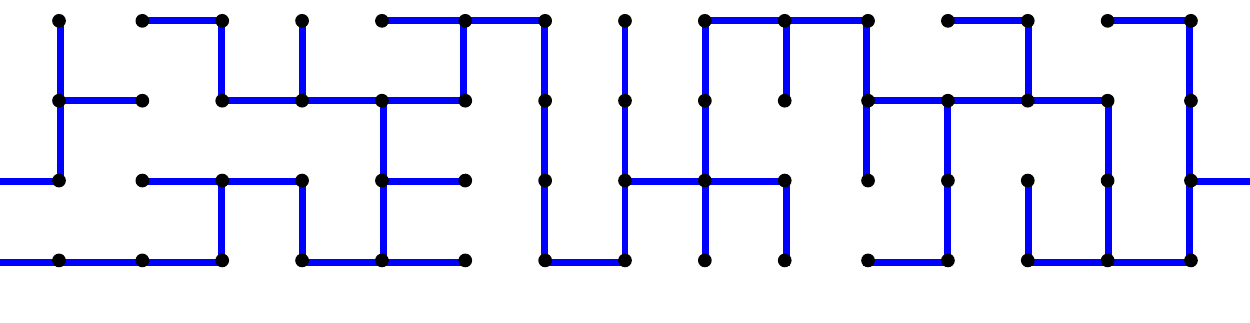}\end{center}
\vskip-.9cm
\caption{\label{esf1}Part of a uniform spanning tree on a strip of width $4$.}
\end{figure}
Figure \ref{esf2} shows random samples from three other measures on the same graph; 
these are measures on \emph{essential spanning forests} (ESFs);
an essential spanning forest is a spanning forest each of whose components
is an infinite tree. These measures are the \emph{locally uniform} measures on ESFs
with $j$ components, for $j=2,3,4$. 

\begin{figure}[htbp]
\begin{center}\includegraphics[width=3in]{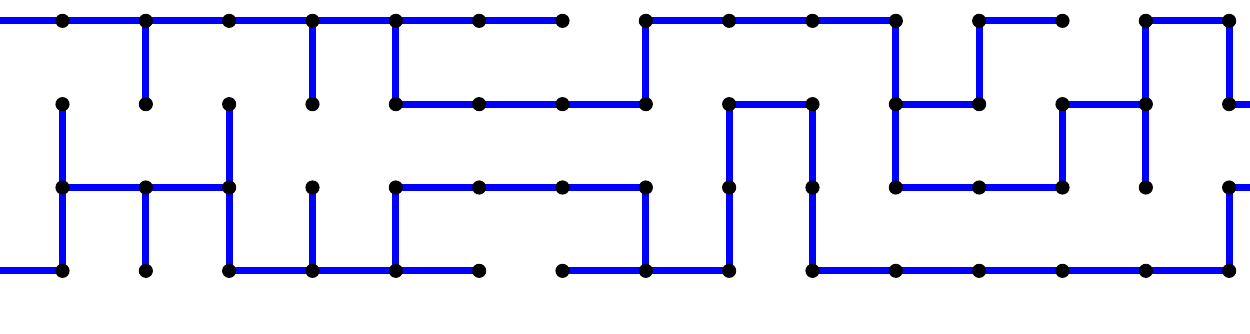}\end{center}
\vskip.5cm
\begin{center}\includegraphics[width=3in]{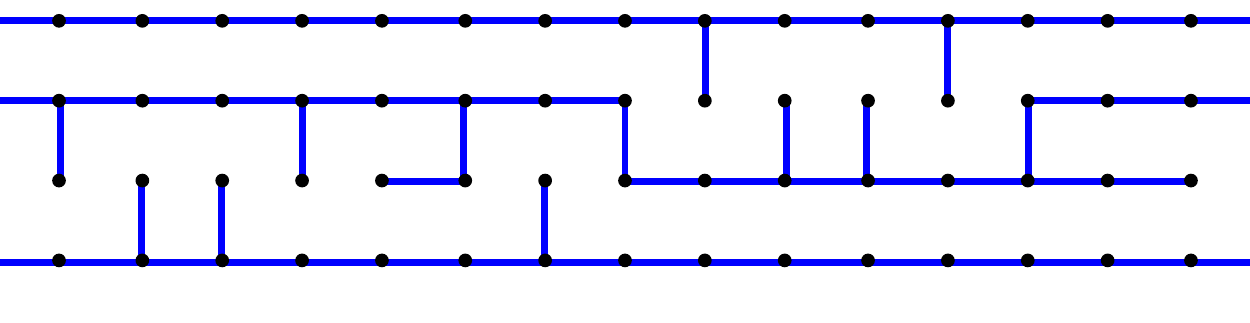}\end{center}\vskip.5cm
\begin{center}\includegraphics[width=3in]{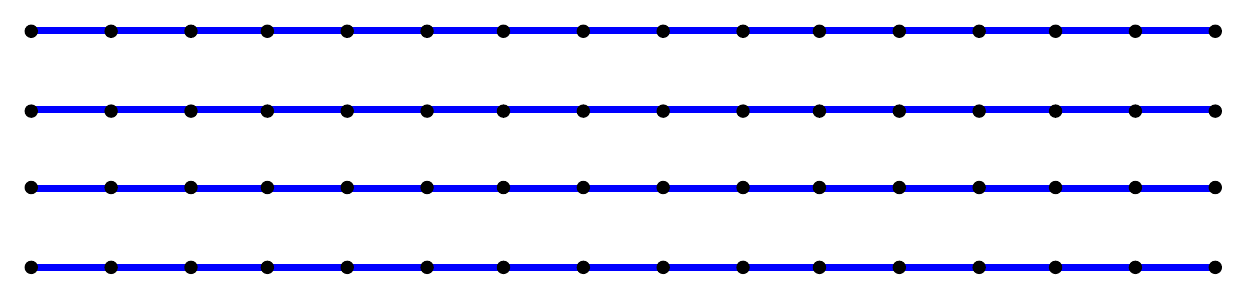}\end{center}
\caption{\label{esf2}Samples of uniform essential spanning forests with $2,3$ and $4$ components.}
\end{figure}
More generally, let $\G$ be a 
planar graph embedded in $\R^2$, invariant under translations in $\Z$ acting by $(x,y)\mapsto(x+n,y)$, 
and with finite quotient $\G_1=\G/\Z$. We call such a graph a \emph{strip graph}.
Let $c$ be a positive function on its edges, the \emph{conductance}, which is also invariant under translations in $\Z$.
We show that there is a unique translation-invariant Gibbs measure $\mu_j$ on ESFs of $\G$ with $j$ components, for all $1\le j\le m$, 
where $m$ is the width of $\G$.
(Here by the term \emph{width} we mean the maximal number of pairwise vertex-disjoint, bi-infinite paths in $\G$). 

We further show that $\mu_j$ is a \emph{determinantal} measure for the edges, that is,
the probability of a given finite set of edges occurring is the determinant of a minor of a certain infinite matrix $K_j$ (the \emph{kernel} of the determinantal measure). 
\begin{theorem}\label{1dthm}
Let $\G$ be a strip graph of width $m$ and conductance function $c$. For each integer $j\in[1,m]$ there is a unique translation-invariant
Gibbs measure $\mu_j$
on essential spanning forests of $\G$ with $j$ components. The measure $\mu_j$ is determinantal for the edges.
The growth rate $a_j$ (or free energy) of $\mu_j$ is
$a_j = \log|C_m| + \sum_{i=j+1}^{m} \log\lambda_i$
where $1<\lambda_2<\dots<\lambda_m$ are the roots larger than $1$ of
the polynomial $P(z) =\det\Delta(z),$
and $C_m$ is its leading coefficient.
\end{theorem}

Here $\Delta(z)$ is the action of $\Delta$ on the space of $z$-periodic functions; see below for the precise definition.
The \emph{growth rate} $a_j$ of $\mu_j$ is defined to be 
the exponential growth rate of the weighted sum of configurations one sees 
in a window of length $n$ as a function of $n$. 

The kernels of the determinantal measures $\mu_j$ 
for different $j$ are simply different Laurent expansions of the same 
finite meromorphic matrix $K(z)$, see (\ref{infkernel}) below.

The statement holds in the more general case of the ``massive'' laplacian determinant, which counts rooted spanning forests with weight 
$M$ per finite component. More generally, let $M$ be a (similarly periodic) 
function which assigns to each vertex $v$ a weight $M_v\ge0$, with at least one weight positive. 
An \emph{$M$-weighted rooted spanning forest} is a rooted spanning forest whose weight is the product of the conductances,
times the product, over all finite components, of the weight of the root of that component.

\begin{theorem}\label{massive1dthm}
Let $\G$ be a strip graph of width $m$ and conductance function $c$, and fix $M$ as above. 
For each integer $j\in[0,m]$ there is a unique translation-invariant Gibbs measure $\mu_j$
on $M$-weighted rooted spanning forests of $\G$, and with exactly $j$ bi-infinite components. The measure $\mu_j$ is determinantal for the edges. The growth rate $a_j$ is
$a_j = \log|C_m|+\sum_{i=j+1}^{m} \log\lambda_i$
where $1<\lambda_1<\dots<\lambda_m$ are the roots larger than $1$ of
the characteristic polynomial $$P(z) =\det(\Delta(z)+D_M),$$
and $C_m$ is its leading coefficient.
\end{theorem}

Here $D_M$ is the diagonal matrix of vertex weights, see below.
See Figure \ref{mt} for examples
with $0$ to $3$ crossings. 
\begin{figure}[htbp]
\begin{center}\includegraphics[width=3in]{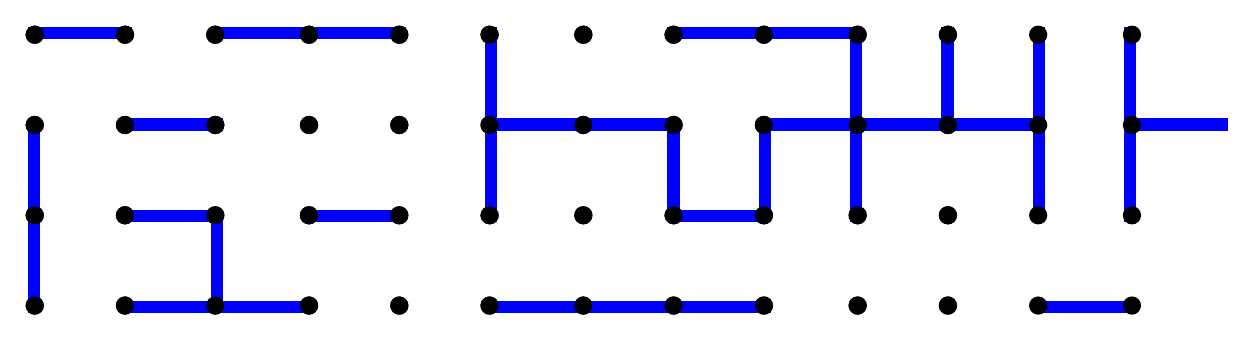}\end{center}\vskip.5cm
\begin{center}\includegraphics[width=3in]{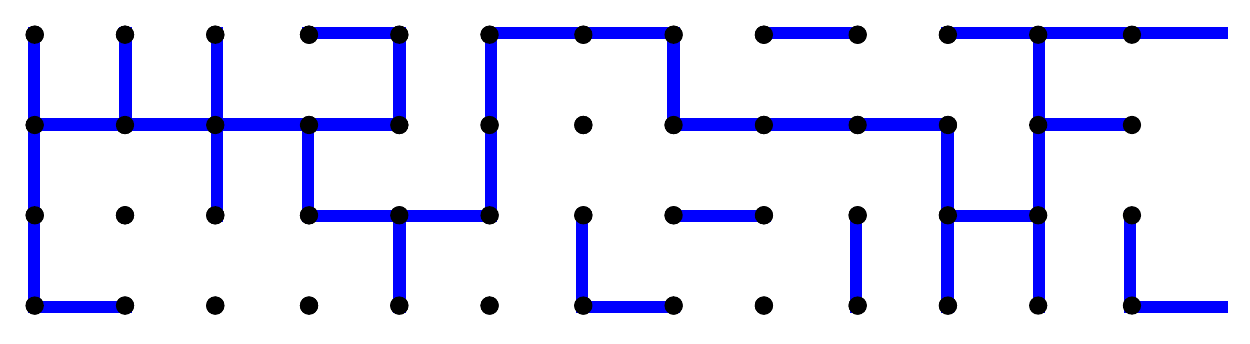}\end{center}\vskip.5cm
\begin{center}\includegraphics[width=3in]{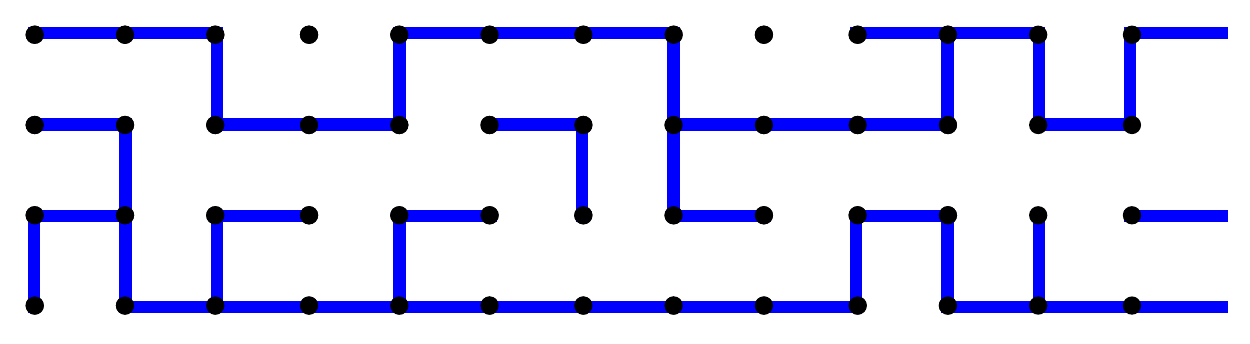}\end{center}\vskip.5cm
\begin{center}\includegraphics[width=3in]{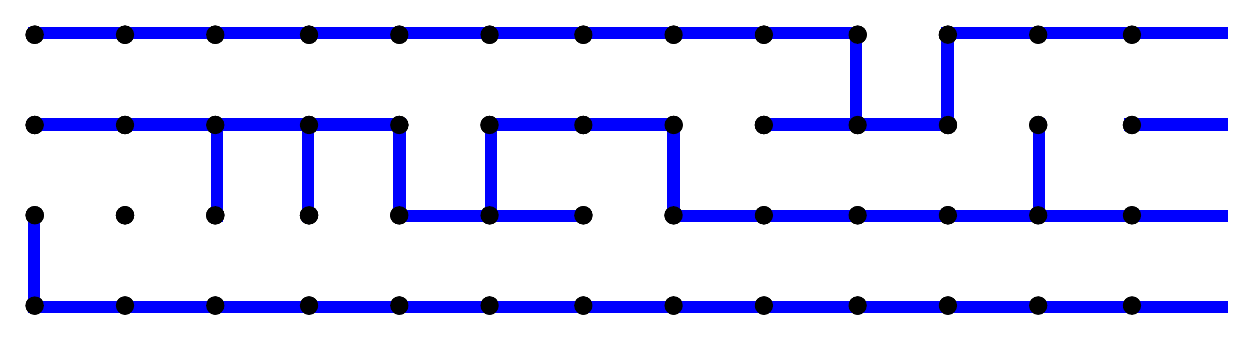}\end{center}
\caption{\label{mt}Exact samples of uniform rooted spanning forests with weight $M=1$ per finite component and with $0,1,2$ and $3$ crossings, respectively.}
\end{figure}
\medskip

The analogs of these measures on doubly-periodic planar graphs are richer.
For a doubly-periodic planar graph $\G$ and a periodic function $M\ge0$
we construct, for any ``realizable'' $(s,t)\in\R^2$, a measure $\mu_{s,t}$
on essential spanning forests of $\G$ (when $M\equiv0$) and $M$-rooted spanning forests (if $M\not\equiv0$, and where only finite components have roots)
where the infinite component trees contain bi-infinite paths with average slope $(s,t)$; here $s$ is the vertical density of paths and $t$ is the horizontal density.
The measure $\mu_{0,0}$ is the spanning tree measure in the case $M\equiv0$ and the $M$-rooted spanning forest measure without infinite components in the other case. 
The realizable slopes $(s,t)$ form a certain convex polygon
$N$, the \emph{flow polygon}, which is also 
the Newton polygon of the so-called characteristic polynomial $P(z,w)$. See Figure \ref{massivebiperiodic} for an example 
on the triangular grid.
We prove the following theorem (for definitions of terms in this statement, see below). 

\begin{theorem}\label{planethm}
Let $\G$ be a $\Z^2$-periodic planar graph in $\R^2$ with $\Z^2$-periodic conductances and periodic vertex weights $M\ge0$. 
Let $P(z,w)=\det(\Delta(z,w)+D_M)$ be the associated characteristic polynomial.
Let $N=N(P)$ be its Newton polygon. For each $(s,t)\in N$ there is a unique translation-invariant Gibbs measure
$\mu_{s,t}$ on $M$-weighted rooted spanning forests of $\G$ (with no finite components if $M\equiv0$),
with infinite components having average slope $(s,t)$
(for $(s,t)=(0,0)$ there is no infinite component if $M\not\equiv0$, and one infinite component if $M\equiv 0$).
The free energy of $\mu_{s,t}$ is the 
Legendre dual of the Ronkin function $R(x,y)$ of $P$. The measures $\mu_{s,t}$ are determinantal for edges.
\end{theorem}

A version of Theorem \ref{planethm} for $M\equiv 0$ was proved independently by W. Sun in \cite{Sun} using Temperley's bijection \cite{KPW}
between dimers and trees. No such bijection is known for $M\not\equiv 0$. Further motivation for studying the $M\ge0$ case comes
from the \emph{conformal invariance} properties of the (scaling limits of the) measures $\mu_{s,t}$ 
for any $(s,t)\in N$, $(s,t)\ne(0,0)$ with
finitely many exceptions. We will not study this conformal invariance here, beyond the correlation decay results below.
 
See Figure \ref{massivebiperiodic} for a sample from the measure $\mu_{s,t}$.
\begin{figure}[htbp]
\begin{center}\includegraphics[width=6in]{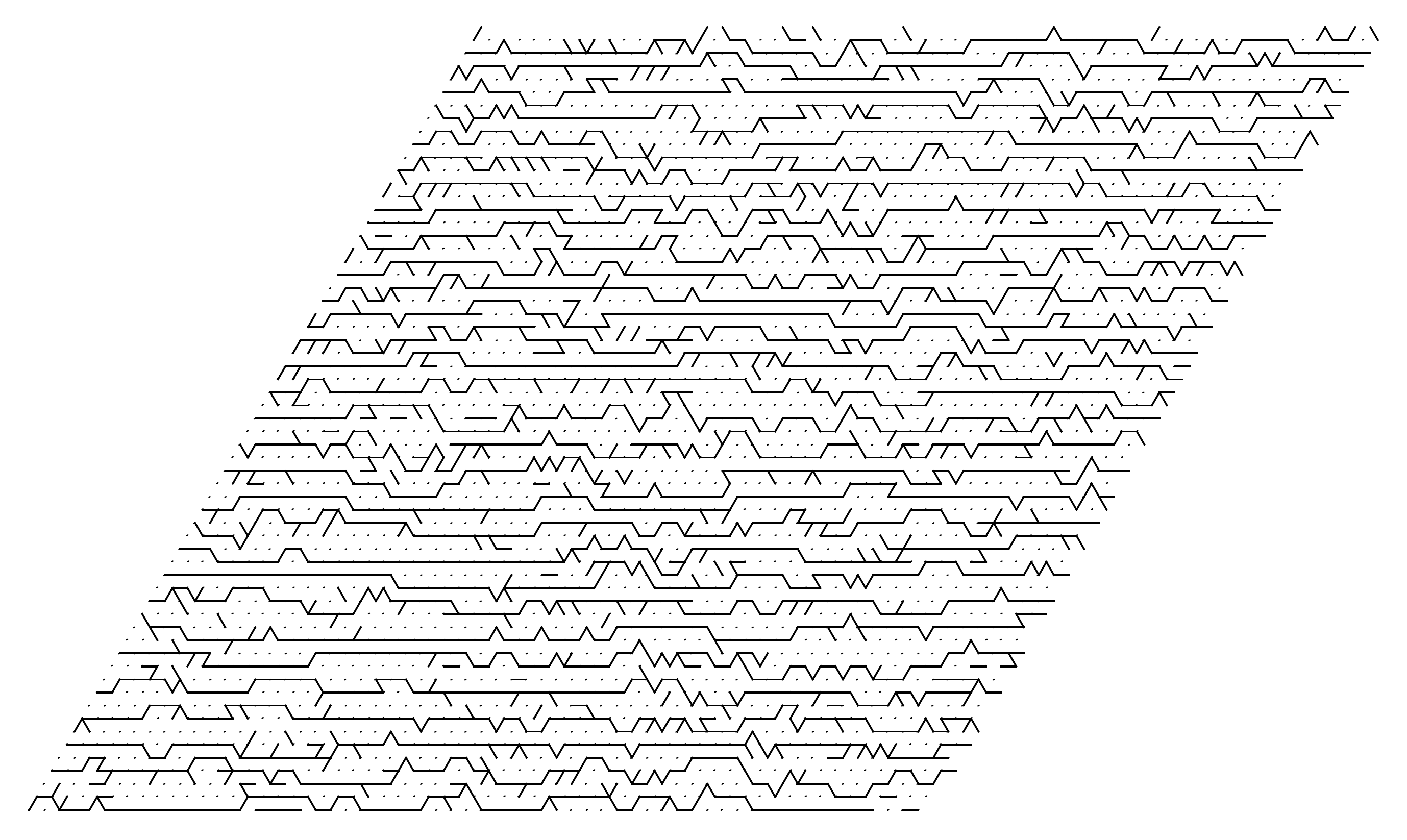}\end{center}
\caption{\label{massivebiperiodic}Sample from the measure $\mu_{0.5,0}$ for constant conductances on the triangular grid with $M\equiv4$.
(We used an MCMC algorithm of unknown mixing time, so this is only an approximately random sample.)}
\end{figure}

An unexpected feature of the measures $\mu_{s,t}$ is that for \emph{integer} slopes $(s,t)\in N$ 
and generic conductances, the decay of edge correlations is exponential in the distance between edges (with one exception,
the case $M\equiv0$ and $(s,t)=(0,0)$).
For $\mu_{s,t}$ for noninteger $(s,t)$ in the interior of $N$,
edges have quadratic decay of correlations. 

\begin{theorem}\label{harnackthm}
In both the massive and massless case, the spectral curve $\{(z,w)~|~P(z,w)=0\}$ is a simple Harnack curve,
symmetric under $(z,w)\to(1/z,1/w)$.
The edge correlation decay is quadratic (in the separation distance between edges)
for noninteger points $(s,t)\in \text{int}(N)$, and exponential
at integer points in $N$, unless the spectral curve $\{P=0\}$ has a real node at a point $(z,w)$
where $\nabla R(\log|z|,\log|w|)=(s,t)$, $R$ being the Ronkin function of $P$.
\end{theorem}

Finally, we consider scaling limits of the essential spanning forest measures on $\eps\G$ when $\eps\to0$.
In particular with fixed boundary connections we consider the following \emph{limit shape} problem (see Figures \ref{limitshapefig}, \ref{limitshapefol}).
\begin{figure}[htbp]
\begin{center}\includegraphics[width=3in]{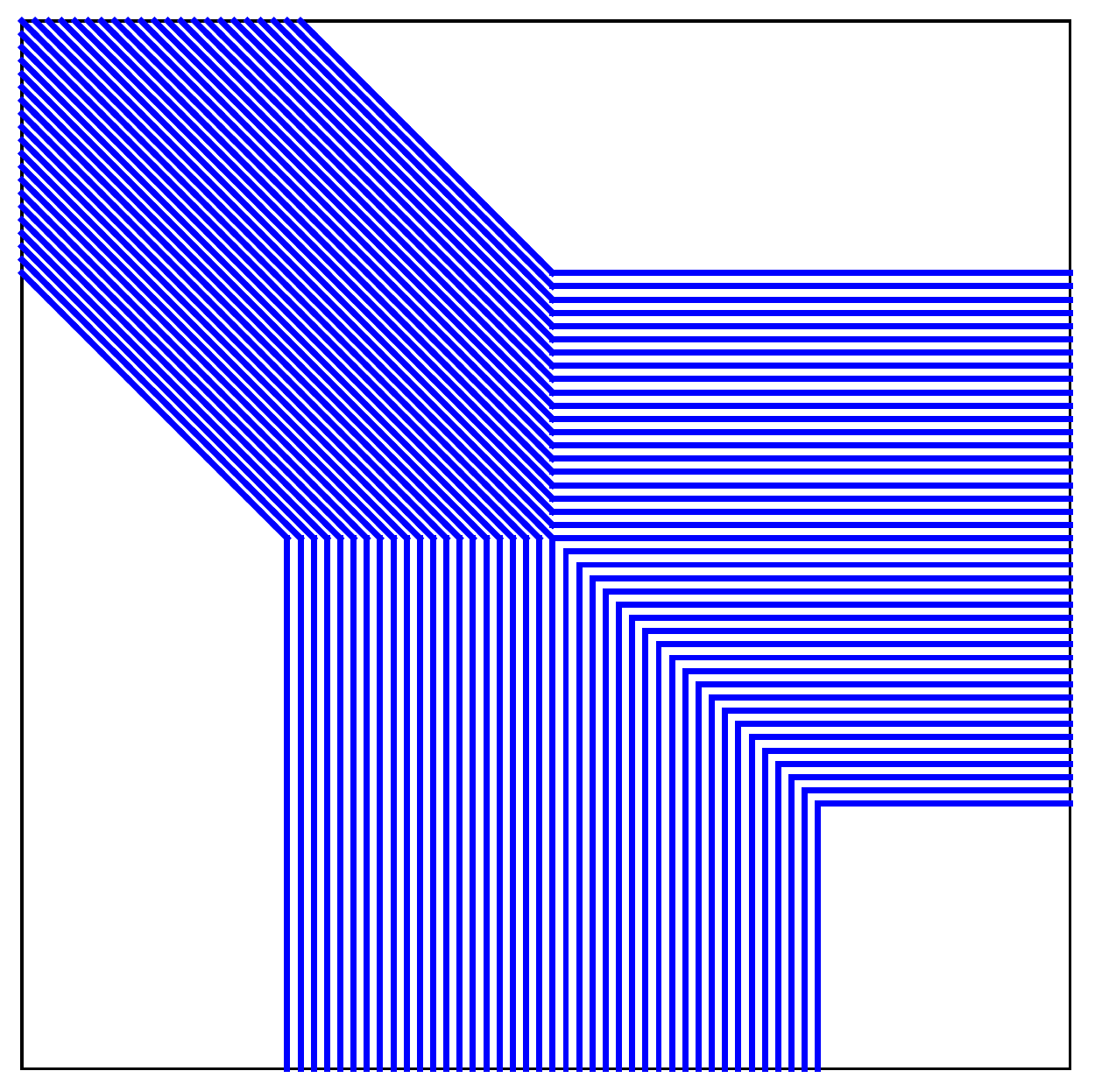}\includegraphics[width=3in]{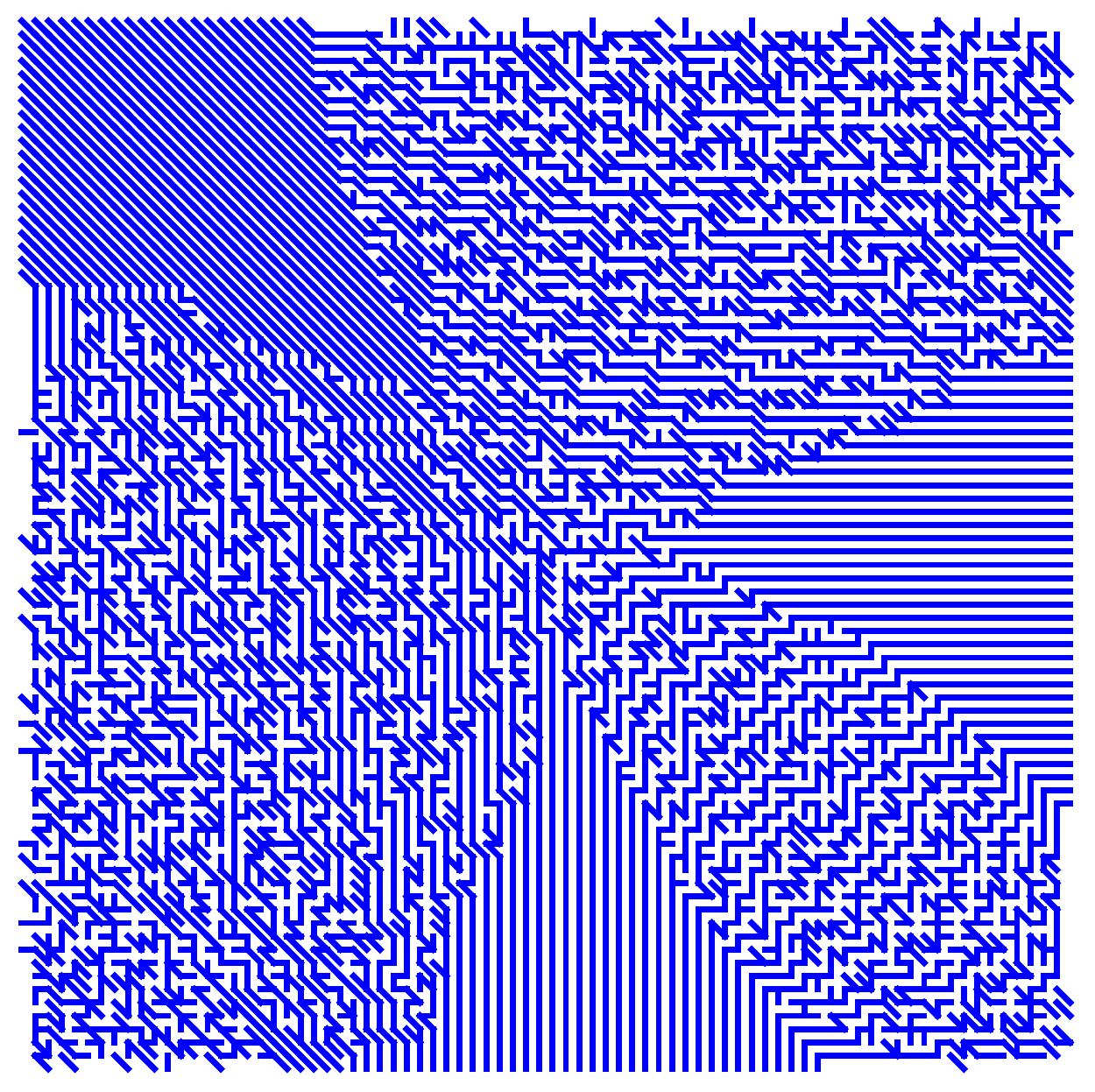}\end{center}
\caption{\label{limitshapefig}The initial grove and a random configuration with the same
boundary connections (the parts of the boundary with no connections have wired boundary conditions).}
\end{figure}
\begin{figure}[htbp]
\begin{center}\includegraphics[width=3in]{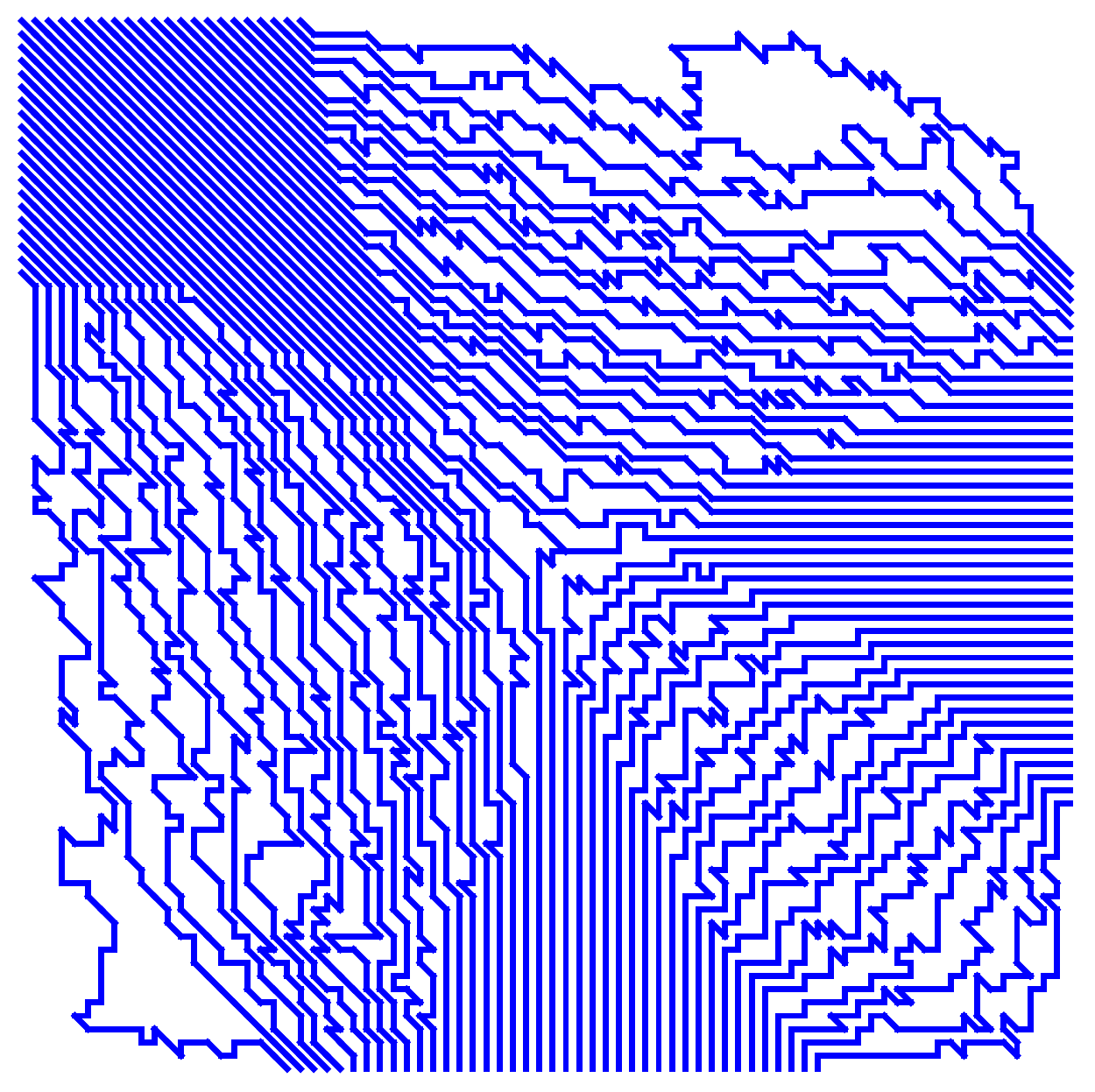}\end{center}
\caption{\label{limitshapefol}The ``trunks" of the random grove of the previous figure. This set of curves approximates the ``grove limit shape", which is a measured foliation.}
\end{figure}
Let $U$ be a simply connected planar domain
with piecewise smooth boundary and let $\cal F$ be a singular measured foliation on $U$ with a finite number of singularities $S$, 
where the measure has transverse derivative in $N$
and leaves transverse to the boundary.
Assume furthermore that all leaves of $\cal F$ begin and end on the boundary.
We approximate $U$ with a sequence of subgraphs $U_\eps\subset\eps\G$ for $\eps>0$. On $U_\eps$ take a random ESF
with component trees isotopic to $\cal F$ in an appropriate sense: the ``trunk'' of each the component tree is isotopic rel $S$ to a leaf of 
$\cal F$ with fixed endpoints on the boundary, and the number of components per unit boundary length approximates the transverse measure of $\cal F$.

Then we prove (see Theorem \ref{limitshapethm} for the exact statement)
\begin{theorem}\label{limitshape} There is a unique (nonrandom) singular measured foliation ${\cal F}_0$ of $U$ with the property
that as $\eps\to0$ a random ESF isotopic to $\cal F$ converges to ${\cal F}_0$: component trees
converge to leaves of ${\cal F}_0$. 
The foliation ${\cal F}_0$ satisfies a variational principle: it minimizes the 
\emph{surface tension} subject to the constraints of being isotopic to $\cal F$. 
\end{theorem} 

Inspiration for Theorem \ref{limitshape} comes from the work of Peterson and Speyer \cite{PS}
who proved a limit shape theorem for ``cube groves". Their result is in fact a special case of Theorem \ref{limitshapethm}, appearing
long before the definitions of the measures $\mu_{s,t}$.

The underlying variational problem of Theorem \ref{limitshape} is algebraically identical to that occurring for certain bipartite dimer models;
as a consequence the minimization equation can be reduced to the complex Burgers' equation, and can therefore
be solved in principle by the method of complex characteristics as in \cite{KO2}. 
However, only for special graphs and special boundary conditions 
has this been worked out explicitly, see \cite{KO2}. 
\bigskip

\noindent{\bf Acknowledgements.}
I would like to thank Robin Pemantle for discussing large parts of the paper with me, and David Speyer for the initial conversations leading to this paper. I thank David Jekel for pointing out an error in a previous version. While this paper was in preparation I had motivational conversations with Wangru Sun who proved Theorem \ref{planethm} for $M\equiv0$ independently in \cite{Sun}.

\section{Background}

For background (beyond what is below) on spanning trees, determinantal measures and their relation to the Laplacian, 
see the modern treatment in \cite{LP}.

\subsection{Trees and measures}

A \emph{spanning forest} of a connected finite graph $\G$ is a collection of edges which has no cycles. 
A \emph{rooted spanning forest} is a spanning forest in which each component has a marked vertex called the root.
A \emph{spanning tree} is a connected spanning forest (connecting all vertices).
If $\G$ has a \emph{boundary}, by which we mean
there is a specified subset $B$ of vertices called boundary vertices (in this paper,
usually a subset of the vertices on the outer face of an embedded planar graph),
then a \emph{grove} of $\G$ (sometimes also called \emph{essential spanning forest}) is a spanning forest each of whose components contains
at least one boundary vertex.
In the case $M>0$, a \emph{massive grove} is a spanning forest of $\G$, each of whose components either has a root (has a marked vertex)
or contains
at least one boundary vertex, or both. Components without roots are called \emph{special} components. 
Note that a grove is a massive grove in which each component is special.
For convenience we unify the terminology and refer to both (massless) groves and massive groves as
simply ``groves" with the massive modifier being understood from the context. 

If $\G$ is infinite, without boundary, by an \emph{essential spanning forest} (ESF) we mean a subset of edges
each of whose components is an infinite tree. An \emph{essential rooted spanning forest} (ERSF) is a subset of edges each of whose 
components is either a finite and rooted tree, or an infinite and unrooted tree. 

For finite $\G$, if $c:E\to\R_{>0}$ is a positive function on the edges (called \emph{conductance}), we define
a probability measure $\ST=\ST_c$ on the set of spanning trees by giving a tree a probability
proportional to the product of its edge conductances: $\nu(T)=\frac1Z\prod_{e\in T}c_e$, where $Z$ is a normalizing constant and $c_e$ is the 
conductance of edge $e$. We call $\ST$ the \emph{spanning tree measure}.
Likewise $c$ defines a probability measure on massless groves by giving a grove a probability proportional to the product of its edge conductances.
For $M\ge0$ a weight function on vertices we define a probability measure on massive groves of $\G$ where a grove $T$ 
has probability $\mu(T) = \frac1Z\prod_{\text{roots $v$}} M_v\prod_{\text{edges $e$}}c_e$ for a constant $Z$.

For an infinite graph $\G$, a \emph{Gibbs measure} on ESFs
is a probability measure on ESFs of $\G$,
with the property that ratios of probabilities of cylinder sets are equal to the ratios of their products of edge conductances,
in the following sense (see \cite{Sheff.ESF}).
Take any finite induced subgraph $\H$ of $\G$, and a spanning tree $T$ of $\G$. Erase the edges of $T$ in $\H$ and 
consider all possible completions of $T$ in $\H$
which have the same connections (within $\H$) between boundary vertices as in $T$. The Gibbs property is that,
conditional on $T$ outside of $\H$,
the probability of any of these completing configurations is proportional to the product of its edge conductances. 

Such Gibbs measures occur as limits of the spanning tree measures or grove measures 
on a growing sequence of finite subgraphs of $\G$ exhausting $\G$.
For some graphs, a limit of spanning tree measures may be supported on spanning trees on $\G$; for other graphs it is supported on ESFs with possibly many components.
For the graphs we consider here (strip graphs and bi-periodic planar graphs) the limit of a tree measure is supported on trees
(the limit of a grove measure may not be, however).

For $M\ge0$ there is an analogous notion of \emph{Gibbs measure on $M$-weighted essential rooted spanning forests}: 
with the same setup as in the above definition of Gibbs measure on ESFs, erase the edges of $T$ in $\H$,
but keep track of any root vertices in $T$, which are considered boundary vertices of $\H$ for the sake of this definition. 
The weight
of any completing configuration of $T$ in $\H$ (having the same connections within $\H$ between boundary vertices as in $T$) is then proportional to the product of its edge conductances times $\prod_vM_v$, the product of weights of the roots (the completing
configuration may have additional roots not present in $T$).

\subsection{Determinantal measures}

A probability measure $\mu$ on $\Omega=\{0,1\}^n$ is \emph{determinantal} if there is an $n\times n$ matrix $K$, the \emph{kernel},
with the following property. Let $S=\{i_1,\dots,i_k\}$ be any subset of $[1,2,\dots,n]$; the event that for a random point of $\Omega$, all indices in $S$ are $1$ is
$\mu(S)=\det(K_S^S)$,
that is, the determinant of the submatrix of $K$ consisting of rows and columns in $S$.
The paradigmatic example is the case when $[n]$ indexes the edges of a connected graph, and $\mu$ is the spanning tree measure,
see Theorem \ref{BPthm} below.

For a determinantal measure, the single point probabilities can also be computed with a similar determinant:

\begin{lemma}\label{singleton}
A measure $\mu$ on $\{0,1\}^n$ is determinantal with kernel $K$ if and only if
for every point $x=(x_1,\dots,x_n)\in\{0,1\}^n$, 
$$\mu(x) = (-1)^{n-|x|}\det(X-K),$$
where $X$ is the diagonal matrix with diagonal entries $(1-x_1,1-x_2,\dots,1-x_n)$ and $|x|=\sum x_i$.
\end{lemma}

\begin{proof} A standard inclusion-exclusion argument, using linearity of the determinant.
\end{proof}

\subsection{Laplacian}

For each edge of a finite graph $\G$ choose arbitrarily one of its two orientations. Let $d:\C^V\to\C^E$ be the corresponding
incidence operator: $df(\vec{uv}) = f(v)-f(u)$. Let $d^*$ be the transpose of $d$ for the standard basis. 
The \emph{laplacian} is defined to be $\Delta = d^*Cd$ where $C$ is the diagonal matrix of conductances.
Concretely
\[\Delta f(v) = \sum_{w\sim v}c_{vw}(f(v)-f(w))\]
where the sum is over neighbors $w$ of $v$.

Define the \emph{transfer current} $K$ to be the operator $K=Cd\Delta^{-1}d^*$. This is well-defined even though $\Delta$ is not in general invertible,
since $\Delta$ can be inverted on the image of $d^*$ and its inverse is unique up to elements in the kernel of $d$.
Note that $K$ is a projection: $K^2=K$. It is sometimes useful to use instead the 
symmetric version $K=C^{1/2}d\Delta^{-1}d^*C^{1/2}$.

Burton and Pemantle \cite{BP} proved that the spanning tree measure $\ST$ is 
determinantal for the edges with kernel $K$:
\begin{theorem}[\cite{BP}]\label{BPthm} For a $\nu$-random spanning tree $T$,
for any $k\ge1$ and edges $e_1,\dots,e_k$ we have $Pr(e_1,\dots,e_k\in T)=\det (K(e_i,e_j)_{1\le i,j\le k}).$
\end{theorem}

\subsection{Bundle laplacian}

For background on material in this section see \cite{Kenyon.bundle}.

Let $\G$ be a graph with conductance function $c:E\to\R_{>0}$. A \emph{connection on a line bundle} (also called
\emph{$\C^*$ local system})
is the data consisting of a $1$-dimensional $\C$-vector space $\C_v$ for each vertex $v$ and, for each edge $e=xy$, 
an isomorphism between the corresponding vector spaces $\phi_{uv}:\C_u\to \C_v$, such that $\phi_{vu}\circ\phi_{uv}=\text{Id}$.
Two connections are \emph{gauge equivalent} if they are related by base change in one or more of the vector spaces $\C_v$.

It is natural to extend the connection to a line bundle over the edges: for each edge $e$ we define a one dimensional $\C$-vector space $\C_e$ and isomorphisms $\phi_{ve}:\C_v\to\C_e$  whenever $v$ is an endpoint of $e$, with $\phi_{ev}=\phi_{ve}^{-1}$ and
$\phi_{ue}\circ\phi_{ev}=\phi_{uv}$. 

Given a closed path $\gamma$ in $\G$ and a vertex $v$ on $\gamma$, the \emph{monodromy} $m(\gamma)\in\C^*$ of the connection
around $\gamma$ is
the isomorphism from $\C_v$ to itself obtained by composing the isomorphisms around $\gamma$: we identify this isomorphism with an element of $\C^*$; as such it is independent of starting point $v\in\gamma$ and only depends on the gauge equivalence class of the connection.

The Laplacian for a graph with connection on a line bundle is an operator 
defined on sections (elements of the total space $\oplus_{v\in V}\C_v$) by 
$$\Delta f(v) = \sum_{w\sim v} c_{vw}(f(v)-\phi_{wv}f(w)).$$

A \emph{cycle-rooted spanning forest} (CRSF) of $\G$ 
is a collection of edges of $\G$ each of whose components has as many vertices as edges,
that is, is a tree with one extra edge, or \emph{cycle-rooted tree}. The weight of a CRSF $\gamma$ is 
$wt(\gamma) = \prod_{e\in \gamma}c_e.$

\begin{theorem}[\cite{Kenyon.bundle}] \label{bundlelap}
On a finite graph $\G$ with connection on a line bundle,
\be\label{bl}\det\Delta = \sum_{\text{CRSFs $\gamma$}}wt(\gamma) \prod_{\eta}(2-m(\eta)-1/m(\eta)),\ee
where the sum is over CRSFs $\gamma$ of $\G$, the product is over cycles $\eta$ of $\gamma$, and $m(\eta)$
is the monodromy of the connection around $\eta$.
\end{theorem}

Note that the cycles in a CRSF are not oriented; in the above expression we need an orientation to compute $m(\eta)$
but the weight $2-m(\eta)-1/m(\eta)$ does not depend on this choice of orientation.

In this paper we consider graphs embedded on surfaces, and we only consider connections
which are \emph{flat}, that is, have trivial monodromy $m=1$ on homologically trivial cycles. Any CRSF on such a graph
with nonzero weight will have only homologically nontrivial cycles.

There is a version of Theorem \ref{bundlelap} for the massive determinant. It follows from Theorem 7 of \cite{Kenyon.bundle}
by adding a vertex to $\G$ connected to every other vertex $v$ by an edge of conductance $M_v$.
Let us define a \emph{multi-type spanning forest} (MTSF) to be a collection of edges in which each component
is either a rooted tree (tree with a distinguished vertex) or a cycle-rooted tree (tree plus one edge, but no root).
The statement of Theorem 7 in \cite{Kenyon.bundle} in this situation is as follows.

\begin{theorem}[\cite{Kenyon.bundle}] \label{bundlelapmass}
On a finite graph $\G$ with connection on a line bundle,
$$\det(\Delta+D_M) = \sum_{\text{MTSFs $\gamma$}}wt(\gamma)\prod_vM_v \prod_{\eta}(2-m(\eta)-1/m(\eta)),$$
where the first product is over the roots of $\gamma$, the second product is over cycles $\eta$ of $\gamma$, and $m(\eta)$
is the monodromy of the connection around $\eta$.
\end{theorem}

\subsection{A linear mapping}

\begin{lemma}\label{linear} For a variable $X$ and constants $C_k$, we have 
$$\sum_{k=1}^m C_k(2-X-X^{-1})^k = \sum_{j=0}^m D_j(X^j+X^{-j})$$
where 
\be\label{Ddef}D_j = (-1)^j\sum_{k=j}^m C_k \binom{2k}{k-j}.\ee
\end{lemma}

\begin{proof} A short induction on $m$.
\end{proof}

\section{Strip graphs and the Laplacian}

\subsection{Characteristic polynomial}
Let $\G$ be a strip graph. Let $\G_1=\G/\Z$ be the finite quotient; it is a finite graph on a cylinder.
On $\G_1$ there is a unique (up to gauge equivalence) flat connection $\phi$ on a line bundle with monodromy
$z$ on a cycle running once around the cylinder in the positive direction. 
Let $\Omega_0(z)$ be the vector space of sections of the bundle. It has dimension $|\G_1|$, the number of vertices of $\G_1$.

Note that $\Omega_0(z)$ can be identified with the vector space of $z$-periodic functions on $\G$, that is, functions $f:\G\to\C$ satisfying
$f(v+1)=zf(v)$ for all vertices $v\in\G$. 

Similarly we define $\Omega_1(z)$ to be the vector space of $1$-forms with values in the line bundle over the edges of $\G$;
a $1$-form is a function $\omega$ on directed edges satisfying $\omega(-e)=-\omega(e).$ 
These $1$-forms can be identified with the vector space of $z$-periodic $1$-forms on $\G$, that is,
$1$-forms $\omega$ on $\G$ which satisfy 
$\omega(e+1)=z\omega(e)$ where $e+1$ represents the edge $e$ translated by $+1$. 

Following \cite{Kenyon.bundle} define $d=d(z)$ the differential $d:\Omega_0(z)\to \Omega_1(z)$ by $df(\vec{xy}) = \phi_{ye}f(y)-\phi_{xe}f(x)$. Then the connection laplacian is  
$$\Delta(z) = \Delta|_{\Omega_0(z)} = d(1/z)^*Cd(z),$$
where $d^*$ is the transpose of $d$ for the standard basis and $C$ is the diagonal matrix of conductances.
All these operators are finite-dimensional. 
Given a fundamental domain for $\G_1$ in $\G$ one can represent $d(z)$ in the standard basis as a matrix 
with entries $0,\pm 1,\pm z,\pm 1/z$, so entries in $\Delta(z)$ are Laurent polynomials in $z$.

Let $P(z) =\det\Delta(z).$
We call $P(z)$ the \emph{characteristic polynomial} of the Laplacian on $\G$. It is a Laurent polynomial in $z$,
of degree $m$ by Theorem \ref{bundlelapmass} (the coefficient of $z^m$ is the weighted sum of MTSFs with the maximal number
$m$ of cycles). 
$P(z)$ is reciprocal: $P(z)=P(1/z)$, because $\Delta(1/z)=\Delta(z)^*$. 
In Theorem \ref{distinct} below we prove that roots of $P$ are real, 
positive and distinct except for a double root at $z=1$; this result first appeared in \cite{Kenyon.EIT} with an incomplete proof.
There are thus exactly $m-1$ roots strictly larger than $1$, where $m$ is the width of $G$.
We let 
$$\lambda_{-m}<\lambda_{-m+1}<\dots<\lambda_{-1}=1=\lambda_1<\lambda_2<\dots<\lambda_m$$ 
be the roots of $P(z)$;
because $P$ is reciprocal we have $\lambda_j\lambda_{-j}=1.$ 
Note that $1$ is a double root and we don't define $\lambda_0$. 

Let 
\be\label{Kzdef}K(z)=Cd(z)\Delta(z)^{-1}d(1/z)^*\ee
be the transfer current operator acting on $\Omega_1(z)$. It is a matrix
indexed by the edges in $\G_1$, with entries which are rational functions of $z$.
\medskip

{\footnotesize By way of example let
$\G$ be the strip graph of Figure \ref{esf1} with conductances $1$. Then $\G_1$ has four vertices; indexing these in order of $y$-coordinate,
we find 
$$\Delta(z) = \begin{pmatrix}3-z-\frac1z&-1&0&0\\-1&4-z-\frac1z&-1&0\\0&-1&4-z-\frac1z&-1\\0&0&-1&3-z-\frac1z\end{pmatrix}$$
and 
$$P(z) = z^4+\frac{1}{z^4}-14 z^3-\frac{14}{z^3}+74 z^2+\frac{74}{z^2}-190 z-\frac{190}{z}+258,$$
with roots
$$\{\lambda_1,\lambda_2,\lambda_3,\lambda_4\} = \{1, 2.11239..., 3.73205..., 5.22274...\}$$
and their inverses.
The kernel $K(z)$ is an asymmetric $7\times 7$ matrix, the first few entries of which are:
$$K(z)=\begin{pmatrix}
-\frac{z(2-5z+2z^2)(1-5z+2z^2)}{(1-4z+z^2)(1-8z+16z^2-8z^3+z^4)}&
\frac{(1-z)^2z}{1-8z+16z^2-8z^3+z^4}&\hdots\\
\frac{(1-z)^2z}{1-8z+16z^2-8z^3+z^4}&\frac{-(1-z)(1-7z+13z^2-7z^3+z^4)}{(1-4z+z^2)(1-8z+16z^2-8z^3+z^4)}&\\
\vdots&&\ddots
\end{pmatrix}.
$$
}

\subsection{Roots of $P(z)$}

\begin{theorem}\label{distinct} Roots of $P(z)$ are real, positive and distinct, except for a double root at $1$.
\end{theorem}

There is a version of this theorem with an incorrect proof in \cite{Kenyon.EIT}. We thank David Jekel for pointing out this error.

\begin{proof}
Since $P(z)=P(1/z)$ and $1$ is a root, it is necessarily a root of even order. 

First let $\G$ be the grid graph $\G_{m,n}$ of width $m$ and length $n$, that is,
obtained from the square grid $\Z\times L_m$ (where $L_m$ is the line graph with $m$ vertices) by
scaling the $x$-axis by $1/n$. Put all conductances equal to $1$. 
In this case, the roots of $P_{m,n}$ are the $n$th powers of roots of $P_{m,1}$ (see (\ref{Pn}) below),
and roots $z_j$ of $P_{m,1}$ satisfy $2-z_j-1/z_j=-X_j$ where the $X_j$ are the eigenvalues of the Laplacian on $L_m$ (see \cite{Kenyon.bundle}),
that is, 
$$X_j=2+2\cos\frac{\pi j}{m}, ~~j=1,\dots,m.$$

Any strip graph of width $m$ 
is a graph minor of $\G_{m,n}$ for some $m,n$, that is, can be obtained from $\G_{m,n}$ by letting some
conductances go to zero (deleting edges) and others to $\infty$ (contracting edges), while maintaining the same width.
We need to show that as we vary the conductances in $[0,\infty]$ the roots of $P$ remain real and distinct. 

For an arbitrary strip graph, let $z$ be a root of $P(z)=\det\Delta(z)$ and let $f$ be a nullvector of $\Delta(z)$. 
We claim that $f$ cannot have a zero on the boundary of $\G$; otherwise let $v$ be a boundary vertex with $f(v)=0$.
If $f$ is zero on all neighbors of $v$, take a path of vertices on which $f\equiv 0$ from $v$ to another vertex $v'$ (with
$f(v')=0$) such that $v'$ has a neighbor on which $f$ is nonzero. Replace $v$ with $v'$. 
By harmonicity, $v$ has a neighbor with positive $f$ value and a neighbor with negative $f$ value;
by repeated use of the maximum principle there is an infinite path starting from $v$ on which $f$ is positive, and an infinite path
on which $f$ is negative. The $\Z$ translates of these paths must be all disjoint from each other by the Jordan curve theorem, contradicting
the fact that $\G$ has finite width. This completes the proof of the claim that $f$ is nonzero on the boundary.
A similar proof shows that $f$ cannot have saddle points (on the boundary or in the interior), that is, points $v$ with four neighbors $v_1,v_2,v_3,v_4$ where in cyclic order where
$f$ is respectively larger, smaller, larger, smaller than $v$.

Since $f$ has no saddle points, on each boundary $f$ is monotone (weakly increasing or weakly decreasing).
Changing sign if necessary, we can assume $f$ is monotone increasing on the lower boundary.
With this normalization, note that on $\G_{m,n}$ the signs of $df$ on the upper boundary for different roots depend on the index of the root, 
alternating from one root to the next,
that is, $f_2$ is decreasing on the upper boundary, $f_3$ is increasing, $f_4$ decreasing,
and so on\footnote{In fact the current $df_i$ has $i-1$ sign changes on a shortest dual path from one side of $\G$ to the other;
our proof shows that this holds for arbitrary strip graphs as well.}. As we vary the conductances this orientation cannot change, otherwise there would
be a set of conductances where the function was constant on the boundary 
(and if $z\ne 1$ this constant must be zero, which is a contradiction).

Suppose that as we vary the conductances two adjacent roots $\lambda_i$ and $\lambda_{i+1}$ merge into a double root,
and consider what happens to $f_i$ and $f_{i+1}$.
If $f_i$ and $f_{i+1}$ converge to \emph{independent} elements of the null space, then a linear combination of them will
be zero at a boundary vertex, a contradiction. If they converge to linearly dependent elements of the null space, both multiples
of a function $f$, then on the upper
boundary $f$ can be neither increasing nor decreasing, that is, must be constant. 
Finally suppose $\lambda_2$ converges to $1$. Rescale $f_2$ to equal $1$ at some vertex $v$ on the lower boundary.
Then values of $f_2$ on the upper boundary are negative, and thus converge to nonpositive values.
On the lower boundary the values converge to $1$. This is a nonconstant harmonic function on $\G_1$,
a contradiction. This proves that $1$ is in fact only a double root.
\end{proof}

\subsection{Growth rate}

\begin{theorem}\label{growth} Let $C_m$ be the leading coefficient of $P(z)$.
The growth rate of the weighted sum of ESFs on $\G$ with $j$ components is $a_i=\log|C_m|+\sum_{i=j+1}^m\log\lambda_i.$
\end{theorem}

\begin{proof}
We compute the characteristic polynomial $P_n$ for $\G_n=\G/n\Z$. ($\G_n$ is again a strip graph invariant under $z\to z+1$ once 
we rescale the horizontal direction by $1/n$.)
By Theorem \ref{bundlelap},
$$P_n(z)=\sum_{\text{CRSFs}~\gamma} wt(\gamma)(2-z-1/z)^{j}$$
\be\label{Nsum}=\sum_{j=1}^m N_j(2-z-z^{-1})^{j(\gamma)},\ee
where $j(\gamma)$ is the number of components of $\gamma$, and $N_j$ is the sum of weights of CRSFs with $j$ components
(each winding once around the cylinder).
We choose $z<0$; then all the terms in the above sums are positive.

Using the translational symmetry of $\G_n$ we can relate $P_n$ with $P=P_1$. We have
\be\label{Pn}
P_n(z) = \prod_{\zeta^n=z}P_1(\zeta) = 
\prod_{\zeta^n=z}\Big[C_m\zeta^{-m}\prod_{\stackrel{j=-m}{j\ne0}}^m(\zeta-\lambda_j)\Big]=
(C_m)^nz^{-m}\prod_{\stackrel{j=-m}{j\ne0}}^mz-\lambda_j^n.\ee
In particular, the roots of $P_n$ are the $n$th powers of the roots of $P_1$.
Now for $i\ge1$ fix $u$ satisfying $\lambda_i< u<\lambda_{i+1}$ (independent of $n$) and take $z=-u^n$.
For large $n$, the terms in the product $z-\lambda_j^n$ are well approximated by $z$ if $j\le i$ and by $-\lambda^n$ if $j>i$.
Thus, taking logs of (\ref{Pn}) and dividing by $n$, as $n\to\infty$ we have 
\be\label{logP}\frac1n\log|P_n(z)|= \log|C_m| + i\log u+\sum_{j>i}\log\lambda_j + o(1).\ee
In the limit $n\to\infty$ 
the RHS of this expression is a convex piecewise linear increasing function of $\log u$, with breakpoints at $u=\lambda_i$.

From (\ref{Nsum}), when $\lambda_i< u<\lambda_{i+1}$, that is, when $\lambda_i^n<-z<\lambda^n_{i+1}$ we have
\be\label{Pnapprox}P_n(z) = \sum_{j=1}^m N_j|z|^je^{o(1)}= \sum_{j=1}^m N_ju^{nj}e^{o(1)}.\ee
From (\ref{logP}) we see that when $\lambda_i< u<\lambda_{i+1}$ 
the leading contribution to $\frac1n\log|P_n|$ 
is the one coming from the $i$th term in (\ref{Pnapprox}); the other terms are exponentially smaller.
Thus the measure concentrates on CRSFs with $i$ components. Moreover for $u$ in this range the sum of weights of
CRSFs with $j$ components is equal to $N_j = P_n/(2-z-1/z)^j$ up to small errors, that
is $|C_m|^n\prod_{i>j}\lambda_i^n.$ Taking logs and dividing by $n$ gives the result.
\end{proof}

\subsection{The kernel $K(z)$}

\begin{theorem}\label{fixedz}
For any $z<0$ or $|z|=1$, $K(z)$ is the kernel of a determinantal measure $\mu(z)$ on the CRSFs of $\G_1$ all of whose components wind
once around the cylinder. 
\end{theorem}

\begin{proof} This proof is essentially taken from \cite{Kenyon.bundle} with minor changes. 
Note that $2-z-1/z>0$ precisely when $z<0$ or $|z|=1$.

Let $\mu(z)$ be the probability measure assigning a CRSF $\gamma$ with $j$ components (all winding once around the cylinder)
a probability $\frac1Zwt(\gamma)(2-z-1/z)^j$, where $Z$ is the normalizing constant
\be\label{Zdef} Z= \det\Delta(z)=\sum_{\text{CRSFs $\gamma$}}wt(\gamma)(2-z-1/z)^j .\ee
Let $e_1,\dots,e_n$ be the edges of $\gamma$. Order the rows of the matrix for $d$ so that the first $n$ edges are $e_1,\dots,e_n$.
Then $d=\begin{pmatrix}d_1\\d_2\end{pmatrix}$ where $d_1$ is $n\times n$ and $d_2$ consists 
of the remaining rows of $d$. Similarly let the diagonal matrix of conductances be $C=\begin{pmatrix}C_1&0\\0&C_2\end{pmatrix}$
where $C_1$ is $n\times n$.

Note that $\det(d_1^*C_1d_1)$ is precisely $wt(\gamma)(2-z-1/z)^j$;
this follows from Theorem \ref{bundlelap} by removing from $\G$ all edges except those of $\gamma$.

Then
\begin{align}\Pr(\gamma) &= \frac{\det(d_1^*C_1d_1)}{\det\Delta}\nonumber \\
&= (-1)^{|E|-n}\frac{
\begin{pmatrix}0&0&C_1d_1\\0&I_{|E|-n}&C_2d_2\\d_1^*&d_2^*&\Delta
\end{pmatrix}}{\det\Delta}\nonumber \\
&= (-1)^{|E|-n}\det\left[\begin{pmatrix}0&0\\0&I_{|E|-n}\end{pmatrix}-Cd\Delta^{-1}d^*\right],\label{oneCRSF}
\end{align}
where in the last equality we used the algebraic identity (where $P,Q,R,S$ are submatrices, with $S$ invertible)
$$\det\begin{pmatrix}P&Q\\R&S\end{pmatrix} = \det S\det(P-QS^{-1}R).$$

A similar computation holds for all other CRSFs $\gamma$, where the matrix $\begin{pmatrix}0&0\\0&I_{|E|-n}\end{pmatrix}$
is replaced by a diagonal matrix with diagonal entries $1$ and $0$, with $0$s in the locations of the edges of $\gamma$. 
Now apply Lemma \ref{singleton} to complete the proof.
\end{proof}

\subsection{Infinite graph}

Now let us consider the infinite strip graph $\G$.

\begin{prop}\label{uniqueness}
There is a unique translation invariant Gibbs measure $\mu_j$ on $j$-component ESFs of $\G$. It is the limit 
as $n\to\infty$ of
Gibbs measures on CRSFs on $\G_n$ with $j$ cycles winding around the cylinder.
\end{prop}

\begin{proof}
The existence of $\mu_j$ follows from the limt of the corresponding measures on $\G_n$ by compactness 
(or use the construction below of $\mu_j$ as a determinantal measure).

For the uniqueness, we use the fact that any allowed local configuration has positive
probability for $\mu_j$. 
Let $H_1$ be a fundamental domain for the translation action on $\G$, and whose removal disconnects $\G$.
If $\mu'$ is another Gibbs measure on ESFs with $j$ components almost surely, 
with positive probability a random sample from 
$\mu_j$ will agree with $\mu'$ on $H_1$. In particular
given a sample from $\mu'$ and a sample from $\mu_j$, 
we can find two integers $n_-<0<n_+$ with $|n_+|,|n_-|$ arbitrarily large, so that
the samples agree on the translated fundamental domains $H_1 + n_+$ and $H_1-n_-$.
The Gibbs property of $\mu_j$ and $\mu'$ implies that on the region between these fundamental domains,
$\mu_j$ and $\mu'$ can be coupled so they agree. Thus $\mu_j$ and $\mu'$ can be coupled to agree on an arbitrarily
large neighborhood of the origin, and so must be equal.
\end{proof}

For two edges $e_1,e_2$ in $\G$, let $[e_1],[e_2]$ denote their images in $H_1$, the fundamental domain for
$\G_1$,
and $x_1,x_2\in\Z$ the translation from $e_1,e_2$ to $[e_1],[e_2]$ respectively.
Then we have the following transfer current formula which defines $\mu_j$.
Recall the definition of the matrix $K(z)$ from (\ref{Kzdef}).

\begin{theorem}\label{infkernel}
For $j\in[1,m]$ define the infinite matrix $K^{(j)}$ by the formula 
\be\label{Kintegral}K_{e_1,e_2}^{(j)}= \int_{\xi_j} K_{[e_1],[e_2]}(u)u^{x_1-x_2}\frac{du}{2\pi i u},\ee
where $\xi_j$ is a circle of radius $r$ for which $\lambda_{j}<r<\lambda_{j+1}$ (or $\lambda_m<r$ if $j=m$).
Then $K^{(j)}$ is the kernel of the determinantal measure $\mu_j$.
\end{theorem}

Note that the RHS depends on $j$ only through the contour of integration. So
the $K^{(j)}$ are simply the different Laurent expansions of $K(z)$. 

\begin{proof} Let $P_n$ be the characteristic polynomial of $\G_n$, and $z<0$.
By Theorem \ref{fixedz} above replacing $\G_1$ by $\G_n$, 
$K_n(z)$ is the kernel of a determinantal measure $\mu(z)$ constructed from $\G_n$.
This $\mu(z)$ is a probability measure on CRSFs of $\G_n$
giving a CRSF $\gamma$ with $j$ components a probability 
$$\mu(\gamma) = \frac1{P_n(z)}(2-z-1/z)^jwt(\gamma).$$

By Theorem \ref{growth}, as $n\to\infty$
for $z<0$ in the range $\lambda_{j-1}^n<|z|<\lambda_{j}^n$, the term $N_j(2-z-1/z)^j$ in the sum (\ref{Nsum}) has larger exponential growth rate
than any of the other terms, 
and so the measure $\mu(z)$ concentrates as $n\to\infty$ on CRSFs with $j$ components.

We can compute $K_n(z)$ for the graph $\G_n$ in terms of $K$ for the graph $\G_1$ as follows. 
Let $e_1,e_2$ be edges of $\G_n$. Let $[e_1],[e_2]$ be their images in a fundamental domain for $\G_1$,
and $x_i$ the translation from $e_i$ to $[e_i]$. Then 
$$[K_n(z)]_{e_i,e_j} = \frac1n\sum_{\zeta^n=z} [K(\zeta)]_{[e_1],[e_2]}\zeta^{x_1-x_2}.$$
In the limit $n\to\infty$ this expression tends to (\ref{Kintegral}); in particular (\ref{Kintegral})
defines a limiting measure $\mu_j$ which is determinantal and supported on
ESFs with $j$ components.
\end{proof}

Note that while the measures for finite $n$ depend on $|z|$, in the limit they do not depend
on $|z|$ in the range $\lambda_j<|z|<\lambda_{j+1}$.

\subsection{Massive case}

Let $\G$ be a strip graph as before.
Let $P_M=\det(\Delta(z) + D_M)$. It is reciprocal: $P_M(1/z)=P_M(z)$, since $\Delta(1/z)^* = \Delta(z)$. 

\begin{lemma}\label{massivereality}
For $M\ge0$ all roots of $P_M$ are real, distinct and positive.
\end{lemma}

\begin{proof}
This is proved in the same manner as Theorem \ref{distinct}.
On $\G_{m,n}$ with conductances $1$, for a constant mass $M\equiv M_0>0$ the roots of $\Delta(z)+D_M$ are real, 
positive and distinct, see the proof of Theorem \ref{distinct}. 
Moreover the corresponding $f_j$, when multiplied by an appropriate sign so that they are increasing on the lower boundary, 
have the property of being alternately increasing and decreasing on the upper boundary as before
(in fact the $f_j$ are independent of $M_0$).

As in the proof of Theorem \ref{distinct}, on $\G$ the $f_j$ cannot be zero on the boundary and cannot have critical points.
As we vary the $M_v$ away from the constant $M_0$, the roots $\lambda_j$ and $\lambda_{j+1}$ cannot merge, since that would lead to a
function $f_j$ with a zero boundary value. Moreover no root can tend to $1$ unless $M\equiv0$ since this would
give a function on $\G_1$ in the kernel of the matrix $\Delta(1)+D_M$ which is nonsingular when at least one $M_v$ is positive.
\end{proof}

\begin{proof}[Proof of Theorem \ref{massive1dthm}.]
The uniqueness of $\mu_j$ follows the same proof as in Proposition \ref{uniqueness} for the massless case.
Take $\G_1$ and add a new boundary vertex connected to all vertices of $\G_1$ with an edge of conductance $M_v$.
We keep to old operator $d$ (ignoring the new edges) and define
$$K(z) = Cd(1/z)^*(\Delta+D_M)^{-1}d(z).$$

Theorem \ref{fixedz} applies with this $K$, although (\ref{Zdef}) is a sum over MTSFs instead of CRSFs, 
but the rest of the proof follows without change.
This shows that $K(z)$ is a determinantal measure on MTSFs on $\G_n$.

Let $1<\lambda_1<\dots<\lambda_m$ be the roots of $P(z)=\det(\Delta(z)+D_M)$ which are larger than $1$.
The proof for the growth rates is nearly identical to the proof of Theorem \ref{1dthm}, except that in (\ref{Nsum}) we are summing over the larger set of MTSFs rather than CRSFs (see Theorem \ref{bundlelapmass}), 
and the range of $j$ values (for the number of infinite components) is $[0,m]$ rather than $[1,m]$. 
The proof of Theorem \ref{infkernel} now extends to this massive case without modification.
\end{proof}

\section{$\Z^2$-periodic graphs}

Let $\G$ be a planar graph embedded in $\R^2$, invariant under translations in $\Z^2$ and with finite quotient. 
Let $\G_1=\G/\Z^2$; this is a graph on a torus. More generally let $\G_n=\G/n\Z^2$.

\subsection{Unit flow polygon}

A \emph{flow} on a graph is a function $\omega$ on oriented edges, 
satisfying $\omega(-e)=-\omega(e)$ (where $-e$ is the edge in the reverse orientation),
and $0=\sum_{u\sim v} \omega(vu).$ Thus the inflow at $v$ (the sum of the $\omega(uv)$ which are positive)
equals the outflow at $v$ (the sum of the $\omega(vu)$ which are positive) at every vertex.

Let ${\cal F}$ be the set of flows on $\G_1$ of capacity $1$ at each vertex, that is,
such that the inflow to each vertex is at most $1$.  
Each such flow $\omega$ defines an element $[\omega]\in H_1(\G,\R)$ and the image of ${\cal F}$ in $H_1(\G,\R)$
is a polytope $X=X(\G_1)$, the \emph{unit flow polytope}.
It is symmetric about the origin in $H_1(\G,\R)$: $X=-X$, because reversing a capacity-$1$ 
flow again gives a capacity-$1$ flow.
It is not hard to see that the vertices of $X$ are integer-valued flows (with value $0,\pm1$ on each edge).

Since $\G_1$ is embedded in $\T^2$ there is an induced linear map $H_1(\G_1,\R)\to H_1(\T^2,\R)\cong\R^2$, 
and the image of $X$ under this map is a convex polygon $N=N(\G_1)\subset\R^2$ with integer vertices, the \emph{unit flow polygon}. 

\begin{lemma}\label{unitflows}
To every integer point in the unit flow polygon $N$ there is a corresponding flow in ${\cal F}$
taking integer values ($0,1$ or $-1$) on the edges, and the nonzero edges form a collection of vertex-disjoint oriented closed loops on $\G_1$. 
\end{lemma}

\begin{proof} Start with a flow $\omega\in X(\G_1)$ with integer homology class $(i,j)\ne(0,0)$. Lift $\omega$ to a flow 
$\tilde\omega$ on
the universal cover $\G$ on $\R^2$.
Such a flow can be represented as $\tilde\omega = \partial h$ (by definition 
$\partial h(e):= h(a)-h(b)$ where $a,b$ are the faces left and right of oriented edge $e$)
where $h$ is a real-valued function on the faces of $\G$. The fact that $\omega$ has homology class $(i,j)$ implies that
for any face $f$,  $h(f+(0,1))=h(f)+i$ and $h(f+(1,0))=h(f)+j$. We can change $\omega$ by adding a boundary 
(so as to not change its homology class) so that $h$ has no extrema: if $h$ 
has a local minimum at a face or connected union of faces (connected across edges), 
increase the value of $h$ on this set of faces to be the minimum of the value on the
neighboring faces. This changes $h$ to $h'=h+h_1$, where $\partial h_1$ has homology $(0,0)$,
and thus does not change the homology class of $\omega$.
Moreover this changes the flow $\partial h$ by decreasing its magnitude at all of the edges bounding the union, and so the new flow remains in $X$. The same operation applies even to infinite collections of faces on which $h$ has a local minimum.

Thus we may assume (up to changing $\omega$ to a different flow in the same homology class)
that $h$ has no extrema. Thus the flow $\partial h$ has no saddle points, that is, at a vertex there are no four edges which in cyclic order have
flow in,out,in,out;
such a saddle point, along with periodicity of $\partial h$, would necessarily lead to an oriented cycle in $\partial h$, which would 
necessarily encircle an extremum for $h$. 

If $\partial h$ is saddle-point-free, the inflowing edges at a vertex $v$ form a contiguous interval in cyclic order around $v$
(ignoring the edges of flow zero) 
and similarly for the outflowing edges. The values of $h$ on the faces
neighboring $v$ in cyclic order are thus decreasing then increasing; that is, they have at most one local minimum and one local maximum.
Since the flow has capacity one all these values are contained in an interval of length $1$. 

Now let $g(f) = \lfloor h(f)\rfloor$,  that is,
$h$ rounded down to the nearest integer. Then $g$ still satisfies  $g(f+(0,1))=g(f)+i$ and $g(f+(1,0))=g(f)+j$,
so $\partial g$ gives an integer flow on $\G_1$ with homology class $(i,j)$. Moreover
$g$ is a unit-capacity flow: this follows from the fact that the values of $h$ on faces neighboring any vertex are
contained in an interval of length $1$.  

The support of $\partial g$, mapped back to $\G_1$ is the desired collection of vertex-disjoint oriented closed loops in homology class
$(i,j)$. 
\end{proof}

Each such collection as in the lemma
can be extended to a CRSF or MTSF on $\G_1$, by removing homologically trivial loops, ignoring the orientation and taking a spanning tree or rooted spanning forest of the complement, wired to the loops.

Conversely we can assign a homology class to a CSRF or MTSF (whose cycles are noncontractible) 
by orienting each loop \emph{in a consistent manner}.
That is,  if one loop has homology class $(i,j)\ne(0,0)$ for one orientation (and thus, since we are on a torus, all other loops have homology class $\pm(i,j)$), 
then we orient all loops so that they
have the same homology class $(i,j)$, giving a total homology class of $(ki,kj)$ if there are $k$ loops. 
Thus to every integer point in $N$ there is an oriented CRSF or MTSF with that homology class, and conversely.

\subsection{Measures}

Let $N$ be the flow polygon of $\G_1$. More generally $N_n=nN$ is the flow polygon of $\G_n$: simply scale up $N$ by the factor $n$.
For each integer homology class $(p,q)\in N_n$, with $(p,q)\ne(0,0)$, write $(p,q)=(ki,kj)$ where $i,j$ are relatively prime and $k=GCD(p,q)$.
Let $\Omega^{(n)}_{p,q}$ be the set of MTSFs on $\G_n$ of total homology class $(p,q)$, that is, with $k$ cycles each 
having homology class $(i,j)$. We let $\mu^{(n)}_{p,q}$ be the associated probability measure on $\Omega^{(n)}_{p,q}$
giving a MTSF a probability proportional to the product of its edge conductances and root weights.
We say $(p/n,q/n)$ is the \emph{slope} of $\mu^{(n)}_{p,q}$; here $p/n$ is the density of cycles per unit length in the vertical direction,
and $q/n$ is the density of cycles per unit length horizontally. 

\begin{theorem}For a point $(s,t)\in N$ the weak-* limit
$$\mu_{s,t} = \lim_{n\to\infty} \mu^{(n)}_{[sn],[tn]}$$
exists and defines a determinantal measure on ERSFs on $\G$ with slope $(s,t)$, that is, 
with infinite components of average direction $s\hat x+t\hat y$
and average density $s$ per unit vertical length and $t$ per unit horizontal length.
\end{theorem}

\begin{proof}
We can view $\mu^{(n)}_{i,j}$ as a (determinantal) measure on ERSFs of $\G$ which are periodic with period $n$. 
For existence of the limit it suffices to show that the kernel of the determinantal measure $\mu^{(n)}_{[sn],[tn]}$ converges.
This is accomplished in Theorem \ref{Kst} below. It remains to show that the limit is supported on ERSFs of slope $(s,t)$.
For this it suffices to show that the cycles for $\mu^{(n)}_{[sn],[tn]}$ do not wander far from their ``average location'', that is,
when measured from the origin one sees the correct density and direction of cycle components. This follows from the 
tail triviality of the limit measure, which is a fact about any determinantal process (see \cite{Lyons}): the horizontal 
and vertical densities of infinite components are tail events and so must have well-defined limits. 
\end{proof}

\subsection{Kernels}

For $z,w\in\C^*$ take a flat line bundle with connection on $\G_1$ having monodromy $z$ on a path with homology $(1,0)$ 
and monodromy $w$ on a path with
homology $(0,1)$. Let $\Omega_0(z,w)$ be the space of sections; as before we can identify $\Omega_0(z,w)$
with the space of $(z,w)$-periodic functions on $\G$, that is, functions $f:\G\to\C$ satisfying
$f(v+(x,y))=z^xw^yf(v)$ for all vertices $v\in\G$ and $(x,y)\in\Z^2$. 
Similarly define $\Omega_1(z,w)$ to be the space of $1$-forms with values in the line bundle over the edges of $\G_1$,
or equivalently, functions $\omega$
on directed edges of $\G$ satisfying $\omega(-e)=-\omega(e)$
and $\omega(e+(x,y))=z^xw^y\omega(e)$ for translations $(x,y)\in\Z^2$.

Define $d:\Omega_0(z,w)\to \Omega_1(z,w)$ as before and 
$$\Delta(z,w) = \Delta|_{\Omega_0(z,w)} = d^*Cd|_{\Omega_0(z,w)}.$$

Let $P(z,w) =\det(\Delta(z,w)+D_M.$
We call $P(z,w)$ the \emph{characteristic polynomial}. 
$P(z,w)$ is reciprocal: $P(z,w)=P(1/z,1/w)$, because $\Delta(1/z,1/w)=\Delta(z,w)^t$.  
In Theorem \ref{harnackthm} (see proof below in Section \ref{Harnacksection}) it is proved that $\{P=0\}$ is a \emph{simple Harnack curve}\footnote{
Among the different definitions/characterizations of simple Harnack curves, the simplest is perhaps that a simple Harnack curve
is the zero set of a real polynomial $P$ with the property that the
intersection of $\{P=0\}$ with any torus
$\{(z,w)\in\C^2~:~|z|=r_1,|w|=r_2\}$ consists in at most two points (and if two points they are complex conjugate points), see \cite{PR}.}.
Let $N(P)$ be the Newton polygon of $P$.

\begin{lemma}
$N(P)=N$ where $N$ is the unit flow polygon of $\G_1$. 
\end{lemma}

\begin{proof}
From Theorem \ref{bundlelap} we have
\be\label{torusPsum}P(z,w) = \sum_{\text{MTSFs $\gamma$}} (2-z^iw^j-z^{-i}w^{-j})^kwt(\gamma),\ee
where the sum is over MTSFs $\gamma$ on the torus graph $\G_1$, $k$ is the number of cycle components of $\gamma$
and $(i,j)$ is the homology class of any such component. The boundary points of $N(P)$ are then precisely the points $z^{ki}w^{kj}=z^pw^q$
where $(p,q)$ are the homology classes of MTSFs on $\G_1$ which are maximal in some direction in homology.
By Lemma \ref{unitflows} these are exactly the boundary points of $N$.
\end{proof}

Let us fix a primitive homology class $(i,j)$ (one with $i,j$ relatively prime). Starting from (\ref{torusPsum}) and expanding, let 
$$P_{i,j}(z^iw^j) = \sum_{k>0} C_{ki,kj}z^{ki}w^{kj}$$ consist of the monomials of $P$ with terms which are powers of $z^iw^j$.
Setting $U=2-z^iw^i-z^{-i}w^{-j}$ we can rewrite $P_{i,j}(z^iw^j)$ as a polynomial in $U$:
$$P_{i,j}(z^iw^j)+c_{i,j}=Q_{i,j}(U) = \sum_{k\ge1} D_kU^k,$$
where $c_{i,j}$ is a constant (not depending on $z$ or $w$) and $D_k$ is the weighted sum of MTSFs with total homology class $(ki,kj)$.
Note that the coefficients $D_k$ are obtained from the $C_{ki,kj}$ via a linear map (Lemma \ref{linear} above).

\subsection{Kernel of $\mu_{s,t}$}

The \emph{Ronkin function} $R(x,y)$ of a bivariate polynomial $P$ is defined as
$$R(x,y) = \int_{|z|=e^x}\int_{|w|=e^y}\log|P(z,w)|\frac{dz}{2\pi i z}\frac{dw}{2\pi i w}.$$
In \cite{Mikh, PR} the following properties of the Ronkin function $R$ are shown.
$R$ is a convex $C^1$ function whose gradient takes values in $N=N(P)$. The map
$\nabla R$ is surjective onto the interior of $N$, and constant on the components of the complement of the amoeba. 
For simple Harnack curves, $\nabla R$ is
nonsingular on the interior of the amoeba of $P$, and maps the interior of the amoeba of $P$
bijectively to the interior of $N\setminus S$, where $S$ is a subset of the integer points in $N$.

We let $\sigma:N\to\R$ be the Legendre dual to $R$:
\be\label{surfacetension1}
\sigma(s,t) = \min_{(x,y)\in\R^2} R(x,y)-s x-t y
\ee
 The surface tension is strictly convex \cite{KOS}.  

Now let $K=K(z,w)=Cd(\Delta(z,w)+D_M)^{-1}d^*$ be the transfer current operator acting on $\Omega_1(z,w)$. It is a matrix
indexed by the edges in $\G_1$, with entries which are rational functions of $z$ and $w$.
For two edges $e_1,e_2$ in $\G$, let $[e_1],[e_2]$ denote their images in the fundamental domain for
$\G_1$,
and $(x_1,y_1),(x_2,y_2)\in\Z^2$ the translations from $e_1,e_2$ to $[e_1],[e_2]$ respectively.
Then we have the following transfer current formula for $\ST$, essentially due to Burton and Pemantle.
\begin{theorem}[Burton and Pemantle \cite{BP}]
$$K_{e_1,e_2}= \iint_{S^1\times S^1} K_{[e_1],[e_2]}(z,w)z^{x_1-x_2}w^{y_1-y_2}\frac{dz}{2\pi i z}\frac{dw}{2\pi i w}.$$
\end{theorem}

We extend this statement in a very simple way, by changing the contour of integration:

\begin{theorem}\label{Kst}
Let $(s,t)$ be a point in the interior of $N$. Let $(x,y)\in\R^2$ be a point which satisfies $\nabla R(x,y) = (s,t)$, where $R$ is the Ronkin function of $P$. Then the kernel $K^{s,t}$ defined by
$$(K^{s,t})_{e_1,e_2}= \iint_{|z|=e^x,~|w|=e^y} K_{[e_1],[e_2]}(z,w)z^{x_1-x_2}w^{y_1-y_2}\frac{dz}{2\pi i z}\frac{dw}{2\pi i w}.$$
is the determinantal kernel for $\mu_{s,t}$ (the limit of the kernels of the $\mu^{(n)}_{[ns],[nt]}$).
\end{theorem}

\begin{proof}
We first discuss the probabilistic meaning of $K(z,w)$. 
On $\G_1$ let 
\be\label{Z}Z = \sum_{\text{MTSFs $\gamma$}}wt(\gamma)(2-z^iw^j-z^{-i}w^{-j})^k\ee
and for any finite set of edges $S=\{e_1,\dots,e_m\}$, let 
$$Z(e_1,\dots,e_m) = \sum_{\text{MTSFs $\gamma$ containing $S$}}wt(\gamma)(2-z^iw^j-z^{-i}w^{-j})^k.$$
We claim that, taking the submatrix of $K(z,w)$ with rows and columns indexed by $S$,
\be\label{Z/Z}\det [K(z,w)_S^S]=\frac{Z(e_1,\dots,e_m)}{Z}.\ee
When $|z|=1=|w|$, and $(z,w)\ne(1,1)$, 
$K(z,w)$ is the determinantal kernel for a probability measure on MTSFs of $\G_1$, where a MTSF $\gamma$ has 
probability proportional to its weight $wt(\gamma)(2-z^iw^j-z^{-i}w^{-j})^k\ge0.$ 
Thus (\ref{Z/Z}) holds when $|z|=1=|w|$. By analytic continuation, the identity of rational functions (\ref{Z/Z}) holds for all
$z,w$, even though these expressions are not necessarily probabilities.

We can write 
\be\label{Zsum}Z=\sum_{(p,q)}D_{p,q}(2-z^iw^j-z^{-i}w^{-j})^k,\ee
where the sum is over nonzero homology classes (with one term for each pair $(p,q),(-p,-q)$), and where
$(p,q)=(ki,kj)$ and $i,j$ are relatively prime. Note $D_{p,q}\ge 0$, since it is a sum of weights of MTSFs.
A similar expression holds for $Z(e_1,\dots,e_n)$, and $D_{p,q}(e_1,\dots,e_m)$ is the sum of weights
of MTSFs containing edges $e_1,\dots,e_m$. 

Fix $(s,t)$ in the interior of $N$ and choose $x,y$ as in the statement.  Assume that $x,y>0$; the other cases are similar.
For each $n$ let $p_0=[ns]$ and $q_0=[nt]$; let $(p_0,q_0)=(k_0i_0,k_0j_0)$ where $k_0=GCD(p_0,q_0)$. 
By adjusting $s$ and $t$ by $o(1)$ we may assume for convenience that $|i_0|,|j_0|$ are of order $n$, that is, $k_0$ is of constant order.
Choose $z,w$ so that $|z|=e^x,~|w|=e^y$ and $z^{i_0}w^{j_0}<0$ where both $z$ and $w$ have small arguments $\arg(z),\arg(w)=O(1/n).$
Now consider the expressions (\ref{Z}) and (\ref{Zsum}) for the graph $\G_n$ rather than $\G_1$.
We claim that as $n\to\infty$, $Z$ concentrates on ERSFs of slope $s,t$ in the following sense. For any $\eps>0$,
in the sum (\ref{Zsum}) the terms $(p,q)$ with $|p-p_0|<\eps n, |q-q_0|<\eps n$ all have approximately the same argument
and dominate all the remaining terms: the sum of the remaining
terms is negligible compared to the sum of these.

The same holds for $Z(e_1,\dots,e_m)$, and the conclusion will follow once we prove this claim.

From the fact that $P=0$ is a simple Harnack curve, it was shown in \cite{KOS} that the coefficients of $P_n$ (which is $Z$ above in (\ref{Zsum}))
satisfy $$[z^{p_0}w^{q_0}]P_n(z,w) = \pm e^{n^2(-\sigma(s,t)+o(1))}.$$
This implies by Lemma \ref{linear} above that $D_{p_0,q_0}=e^{n^2(-\sigma(s,t)+o(1))}$
(because the coefficients in Lemma \ref{linear} are exponential in only \emph{linear} functions of $n$, the sum (\ref{Ddef})
in this setting is dominated by the first term).
Since $\sigma$ is the Legendre dual of $R$, and is strictly convex, the sum of the terms 
$$|D_{p,q}(2-z^iw^j-z^{-i}w^{-j})^k| = |D_{p,q}z^{p}w^{q}|(1+o(1))$$ 
with $\|(p,q)-(p_0,q_0)\|<\eps n$ 
has larger exponential growth rate
than any of the other terms in (\ref{Zsum}). So for $n$ sufficiently large they dominate the sum (\ref{Zsum}).
Moreover these terms have all approximately the same argument:
$$\arg(z^pw^q) = \arg(z^{p_0}w^{q_0}z^{O(\eps n)}w^{O(\eps n)}) =\arg(z^{p_0}w^{q_0})+O(\eps)$$
by our condition on $\arg z,\arg w$. 
So we can remove the absolute values and conclude that
the sum of the terms $D_{p,q}(2-z^iw^j-z^{-i}w^{-j})^k$ with $\|(p,q)-(p_0,q_0)\|<\eps n$ dominates the sum (\ref{Zsum}). 

Thus with probability tending to one the configuration concentrates on ERSFs of slope $(s,t)$. 
\end{proof}

\section{Harnack property and edge correlations}\label{Harnacksection}

In this section we prove Theorem \ref{harnackthm}. 

\subsection{Minimal graphs and $Y-\Delta$ transformations}

Two planar graphs $\G,\G'$ with edge conductances and vertex weights are \emph{electrically equivalent} if they are related by a sequence of \emph{electrical transformations}, which are local rearrangements of the graph, 
see Figure \ref{ETs} and \cite{KW}. For example in the first transformation, one can remove a vertex of degree two, replacing it with an edge between its neighbors and update the
weights and conductances as shown. All these transformations preserve the measure on rooted spanning forests,
in the sense that there is a local weight-preserving mapping from rooted spanning forests of the ``before" graph with those on the 
``after" graph. 

\begin{figure}[htbp]
\begin{center}\includegraphics[width=1.5in]{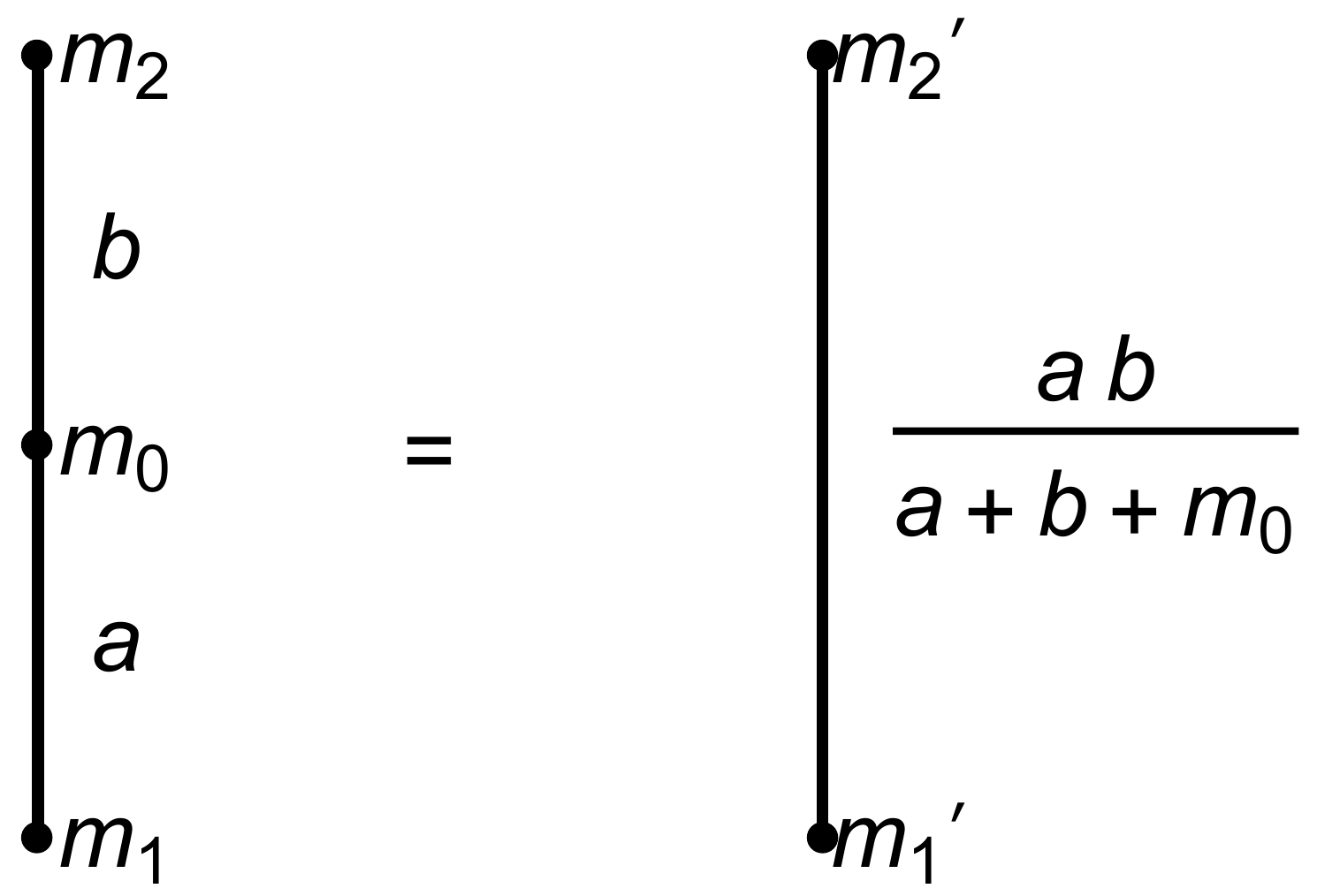}\hskip1in\includegraphics[width=1.5in]{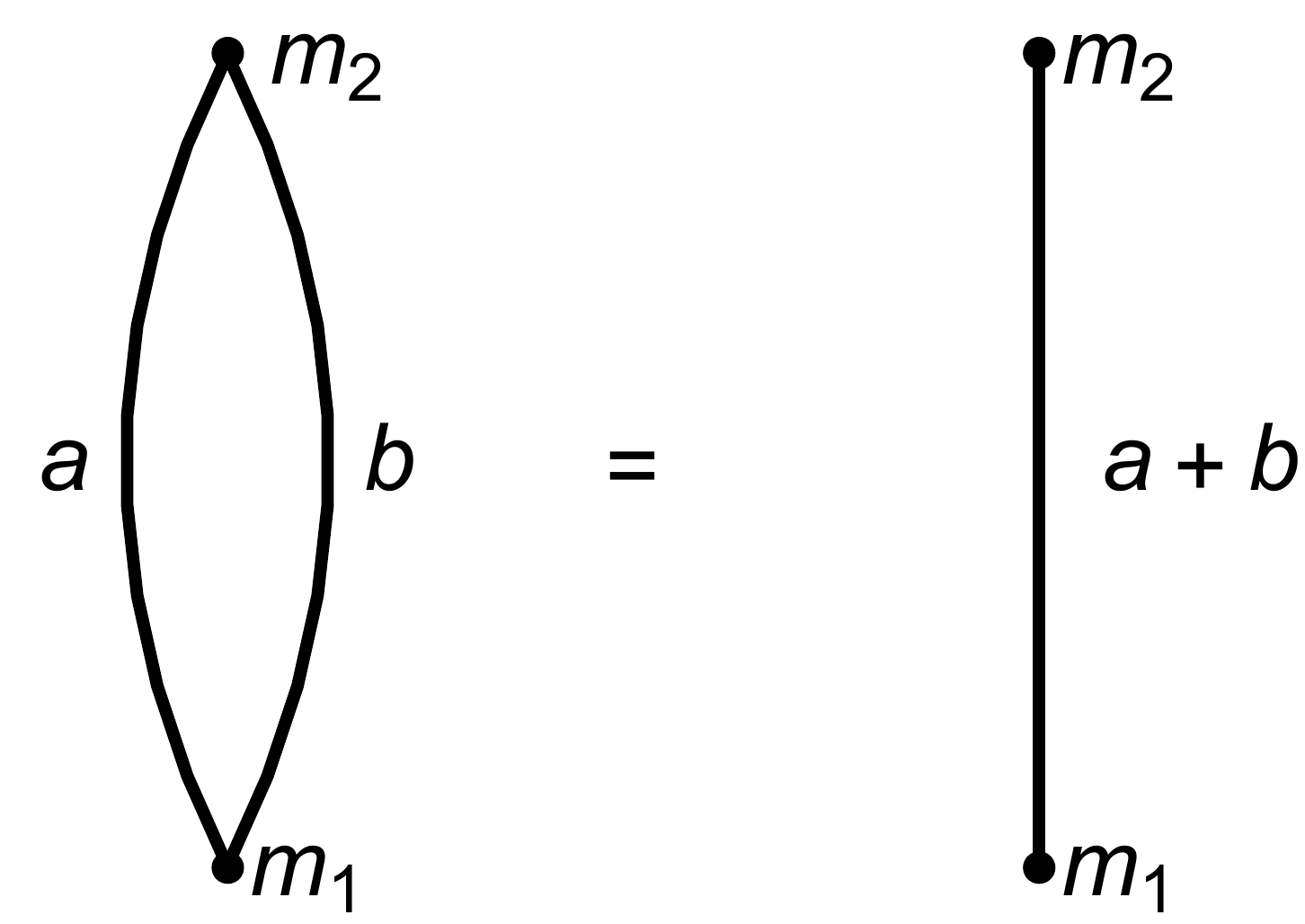}\end{center}
\begin{center}\includegraphics[width=1.5in]{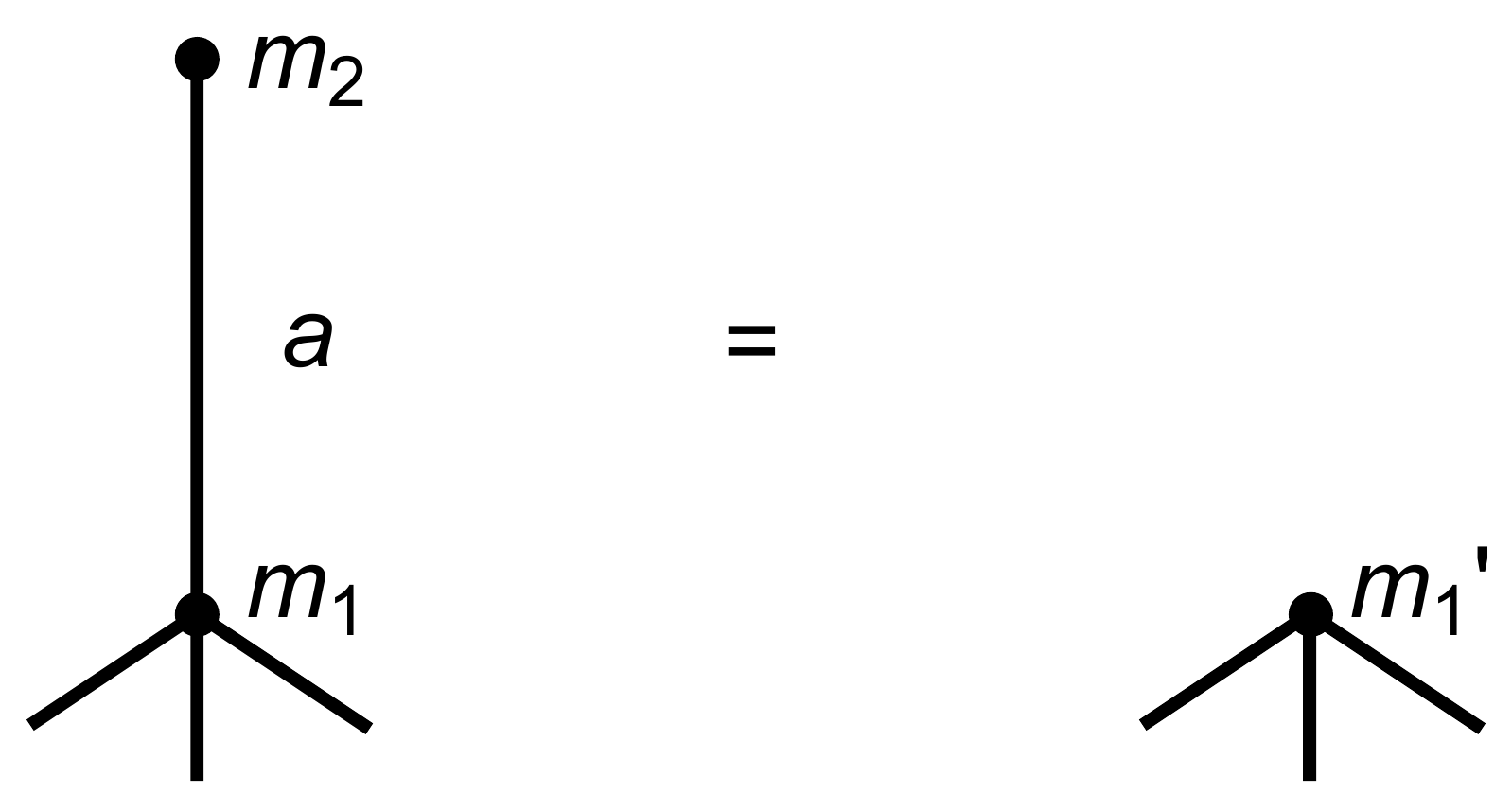}\hskip1in\includegraphics[width=2.5in]{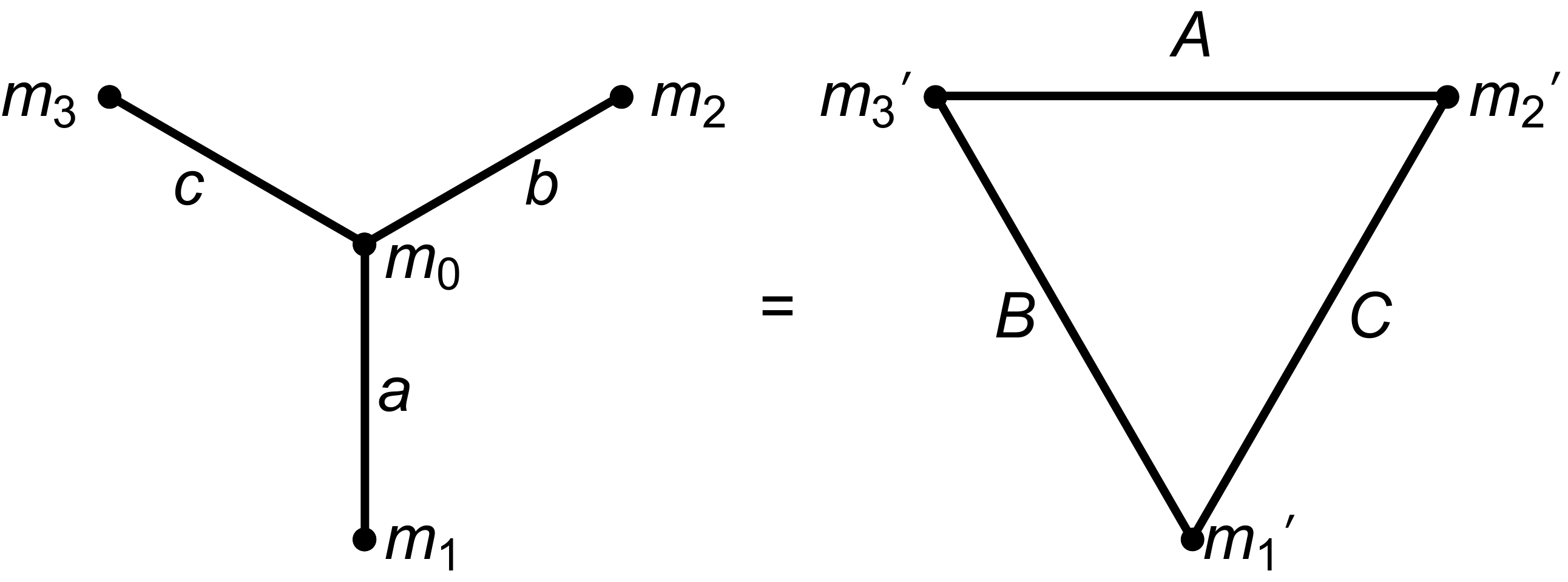}\end{center}
\caption{\label{ETs}Massive electrical transformations: series move, parallel move, removal of dead branch, and star-triangle transformation. 
Here the $a,b,c,A,B,C$ are conductances and the $m_i$ are
the vertex weights. For the first transformation we have
$m_1' = m_1+\frac{am_0}{a+b+m_0}$ and $m_2'=m_2+\frac{bm_0}{a+b+m_0}$. 
For the third transformation we have $m_1' = m_1+\frac{am_2}{a+m_2}$.
For the fourth transformation we have $A=\frac{bc}{a+b+c+m_0}, m_1' =m_1+\frac{am_0}{a+b+c+m_0}$ and symmetrically for $B,C,m_2,m_3$. }
\end{figure}

A graph on a torus is said to be \emph{minimal} if it has the fewest edges in its electrical equivalence class.
Minimal graphs on a torus were characterized in \cite{GK} as those having the property that, on the cover $\tilde\G$ in $\R^2$,
the zig-zag paths of $\tilde\G$ are not closed loops, do not self-intersect, and two zig-zag paths intersect at most once. 
(Even though in \cite{GK} we considered only the case $M\equiv 0$, the definition and characterization of minimal
graphs is topological and so extends to the current case as well.)

Electrically equivalent graphs on the torus have the same characteristic polynomial $P(z,w) = \det(\Delta(z,w)+D_M)$, up to a
multiplicative constant, hence the same spectral curve $P=0$. In particular when studying spectral curves of graphs on a torus,
it suffices to consider minimal graphs.

\begin{lemma}\label{bdynoM} For a minimal graph $\G$ on a torus the coefficients of $P(z,w)$ on the boundary $\partial N$ 
do not depend on $M$.
\end{lemma}

\begin{proof}
In \cite{GK}, which dealt with the case $M\equiv0$, it was shown that for minimal graphs, for every boundary point on $N$, all corresponding CRSFs have a cycle
passing through every vertex. Thus no MTSF which has roots can have a homology class on $\partial N$. By (\ref{torusPsum}), the boundary coefficients of $P$ only arise from CRSFs, not MTSFs, and thus do not depend on $M$.
\end{proof}

\subsection{Proof of Harnack property}

We prove Theorem \ref{harnackthm}. This proof strategy parallels the proof in \cite{KO1} for spectral curves in the dimer model.

We need a few facts about simple Harnack curves, which can be found in \cite{Mikh}. 
First, simple Harnack curves $P=0$ with given Newton polygon $N$
are characterized by having the maximal
number, $g+1$, of real components (one for each of the $g$ interior integer points of $N$, plus one), and the condition
that these components have a certain topological arrangement in $\R P^2$: none is surrounded by any other except for the one corresponding to the boundary of $N$ which
surrounds all others. Simple Harnack curves with given Newton polygon $N$
are the closure of an open set in the space of all real affine curves (defined by real polynomials with Newton polygon $N$), whose boundary
consists only in the following types of degenerations: when an oval (a real component) shrinks to a point, or two 
points at $\infty$ meet, that is, one of the single-variable polynomials corresponding to an edge of $N$ has a double root.  
The set of all simple Harnack curves with all possible Newton polygons is also connected: one can degenerate a simple Harnack
curve by (after scaling so that the largest coefficient stays bounded) sending certain boundary coefficients of $P$ to zero, as a result of which certain real components will move off to $\infty$,
merging with the outer real component.
The limit will be a simple Harnack curve with smaller $N$. 

Let $\G$ be a $\Z^2$-periodic planar graph with $\Z^2$-periodic conductances.
Let $P(z,w) = P_M(z,w) = \det(\Delta(z,w)+D_M)$ be the characteristic polynomial. 
Note that $P$ is reciprocal since $\Delta(z,w)^t = \Delta(1/z,1/w)$.
When $M\equiv0$, the polynomial $P(z,w)$ was shown to define a simple Harnack curve $\{P=0\}$ in \cite{GK},
with a real node at $(z,w)=(1,1)$. 
We show that as we increase each $M_v$, the curve remains simple Harnack. Since the boundary points of $\{P=0\}$ 
only depend on the coefficients of $P$ on the boundary of $N$, and these do not depend on $M$ by Lemma \ref{bdynoM}, 
the curve cannot cease to be Harnack due to collisions of points at infinity.
Therefore it suffices to show that the ovals
do not disappear (as we vary among Harnack curves the ovals never meet each other). 

The oval at the center plays a different role than the other ovals; we deal with this oval first.
Suppose, starting from the simple Harnack curve when $M\equiv0$, we increase at least one $M_v$. 
Then we claim the real node at $(1,1)$ becomes
an oval, that is, $\{P=0\}$ no longer intersects the unit torus $\{|z|=|w|=1\}$.  
Suppose on the contrary we have a zero $(z,w)=(e^{i\theta},e^{i\phi})$. We then have 
a quasiperiodic real function $f$ on $\G$ in the kernel of $\Delta(z,w)+D_M$. If both $z,w$ are roots of unity, say $p$th and $q$th roots of unity respectively,
then on the $pq$-fold torus cover $\G_{pq}$, the function $f$ descends to a function $\tilde f$ in the kernel of $\Delta+\tilde D_M$, the 
\emph{standard} laplacian on $\G_{pq}$. 
This is impossible since this laplacian is invertible as long as some $M_v>0$. 
So we may suppose that at least one of $z,w$ is not a root of unity. From each 
vertex $v\in\G$ for which $f(v)>0$, there is a path to $\infty$ on which $f$ is increasing (that is, nondecreasing and not eventually constant).  This is impossible by quasiperiodicity of $f$, that is, by the fact that $f(v+(j,k))=z^jw^kf(v)$ for integer $(j,k)$.
This shows that for $M\not\equiv0$, the oval at the origin of $N$ does not close.

Now let us deal with the other ovals. On $\{P=0\}$, the matrix $\Delta(z,w)+D_M$ is singular; generically it has corank $1$,
and thus its cofactor matrix $Q=(\Delta(z,w)+D_M)^*$ has rank $1$. Consider the first column of $Q$; 
its entries $Q_{1i}$ are Laurent polynomials in $z,w$. At any common
zero of $P$ and $Q_{1i}$, either the entire first column of $Q$ vanishes or the $i$th row of $Q$ vanishes (due to the fact that $Q$ has rank $1$).
The points $(z,w)$ where the first column vanishes form a divisor on $\{P=0\}$, called a \emph{special divisor} in \cite{KO1}. 
We claim that
the special divisor consists of $g-1$ real points, one point on each oval of $P$, except for the central oval. 
To see this, check first in the case $\G$ is the square grid with unit conductances, and $M\equiv 0$: see Lemma \ref{harmonicquadlemma} below.
Our desired graph is a graph minor of the square grid: it can be obtained from the grid by deletions and contractions of edges. 
We can thus deform the square grid to our desired $\G$ by changing the conductances in $[0,\infty]$, with conductance
$0$ corresponding to deleting an edge and conductance $\infty$ corresponding to contracting an edge.
As we change conductances in $(0,\infty)$
the number of solutions to $Q_{11}=\dots=Q_{1n}=0$ does not change, and, since for simple Harnack curves the ovals never meet,
there remains one solution on each oval: as a consequence the curve is still simple Harnack.
Now let certain conductances degenerate to $0$ or $\infty$; in this case the curve $\{P=0\}$ can degenerate: 
some of its exterior coefficients can tend to zero,
and the Newton polygon will change. Certain real components will move off to infinity and merge with the exterior component. 
But the remaining bounded real components
will still not touch each other (by the Harnack property) and thus will still contain one real divisor point each.   
Once we have deformed the initial square grid to our desired graph $\G$ with desired conductances, we increase $M$ away from $M\equiv0$. 
Again there remains one divisor point on each oval, so the number of real components to $P=0$ can not decrease.
Therefore the curve $P$ remains simple Harnack. 
 
As a consequence of the fact that $P$ is simple Harnack, $P=0$ intersects the torus $\T=\{(z,w)~:~|z|=e^x,|w|=e^y\}$ 
either transversely at two conjugate points $(z_0,w_0),(\bar z_0,\bar w_0)$, or at a real node, or at a real point on the boundary of
the amoeba of $P$, or not at all. In the first two cases the Fourier coefficients $C_{x,y}$ of $1/P$, defined by
\be\label{FCs}C_{x,y} = \frac1{4\pi^2}\int_{(z,w)\in \T} \frac{z^x w^y}P\,\frac{dz}{iz}\,\frac{dw}{iw}\ee 
for $(x,y)\in\Z$, decay linearly as $|x|+|y|\to\infty$, 
and in the last two cases they decay geometrically,  see \cite{KOS}.
For edges $e_1,e_2\in \G_1$, by Theorem \ref{Kst} above
the kernel $K_{e_1,e_2+(x,y)}$ for $(x,y)\in\Z^2$ is a linear combination of the Fourier coefficients $C_{(x,y)+(j,k)}$ for a finite set $(j,k)$, 
with coefficients independent of $(x,y)$. Thus $K_{e_1,e_2+(x,y)}$ decays correspondingly as $|x|+|y|\to\infty$.
This implies that the edge-edge covariances decay quadratically in the first two cases and exponentially fast in the last
two cases. 
This completes the proof. \hfill{$\square$}

\begin{lemma}\label{harmonicquadlemma}
For the square graph $\Z^2/n\Z^2$ with unit conductances the special divisor has one point on each real component
of $\{P=0\}$ except for the central component.
\end{lemma}

\begin{proof}
The characteristic polynomial is 
\be\label{Pproduct}P_n(z,w) = \prod_{\xi^n=z}\prod_{\eta^n=w} 4-\xi-\frac1{\xi}-\eta-\frac{1}\eta.\ee
Points on $P_n=0$ correspond to $(z,w)$-periodic harmonic functions on $\Z^2$, that is, harmonic functions $f$
with the property $f(x+n,y)=zf(x,y)$ and $f(x,y+n)=wf(x,y)$. 
Such a function has the form $f(x,y) = \text{Re}(C\xi^x\eta^y)$ where $4-\xi-1/\xi-\eta-1/\eta=0$, 
$(\xi^n,\eta^n)=(z,w)$, $z,w$ are real and $C$ is a constant. There is a two-dimensional space of such functions, obtained by varying $C$.
By choosing the argument of $C$ appropriately we can make $f$ zero at any desired vertex,
e.g. vertex $1$, as long as $\xi,\eta$ are not both real (if they are both real then $f$ is identically zero). Making $f$ zero at the first vertex corresponds to making the first column of $Q$ vanish.
Note that if $(z,w)=(1,1)$, then $4-\xi-1/\xi-\eta-1/\eta=0$ only when $(\xi,\eta)=(1,1)$ are both real, so this case is disallowed.

So it suffices to find the remaining number $g-1=2n^2-2n$ of real points $(z,w)\ne(1,1)$ on $P_n=0$. 
From equation (\ref{Pproduct}) we must thus find $(\xi,\eta)$, roots of $4-\xi-1/\xi-\eta-1/\eta=0$, 
with arguments which are multiples of $\pi/n$.

Note that $4-\xi-1/\xi-\eta-1/\eta=0$ has the rational parametrization
$$\xi=\frac{(u+a)(u+b)}{(u-a)(u-b)}~~~~~\eta=\frac{(u+a)(u-b)}{(u-a)(u+b)},$$
where $u\in\C$, $a=e^{\pi i/4}$ and $b=e^{3\pi i/4}$.
In order to find $(\xi,\eta)$ with arguments $\theta,\phi$ respectively,
we need to find $u$ such that 
\begin{align}\arg(\xi\eta) &= \arg\frac{(u+a)^2}{(u-a)^2} = \theta+\phi\label{argeqns1}\\
\arg(\xi/\eta) &= \arg\frac{(u+b)^2}{(u-b)^2} = \theta-\phi.\label{argeqns2}\end{align}

The set of points $u\in\C$ for which $(u+a)/(u-a)$ has fixed argument is a circle going through $a$ and $-a$ and with center on the line $x+y=0$.
Similarly the set of points $u$ for which $(u+b)/(u-b)$ has fixed argument is a circle going through $b,-b$ and with center on the line $x-y=0$.
For $u$ in the open unit disk, $\text{arg}(\frac{u+a}{u-a}),\text{arg}(\frac{u+b}{u-b})\in(\frac{\pi}2,\frac{3\pi}2)$.
Taking $u$ on the boundary of the unit disk leads to $\xi,\eta$ both real
so we can ignore this case. 
Thus given $\theta,\phi$, there is a unique point $u$ in the disk
satisfying (\ref{argeqns1}),(\ref{argeqns2}). As $\frac{\theta+\phi}2,\frac{\theta-\phi}2$ run over the $2n-1$ multiples of 
$\pi/2n$ in $(\frac{\pi}2,\frac{3\pi}2)$, $\theta$ and $\phi$
take $(2n-1)^2$ possible values, of which $2n^2-2n+1$ are multiples of $\pi/n$. Removing the solution $u=0$ when both angles are zero, we have $2n^2-2n$ solutions.
See Figure \ref{harmonicquads}.
\end{proof}

\begin{figure}[htbp]
\begin{center}\includegraphics[width=2in]{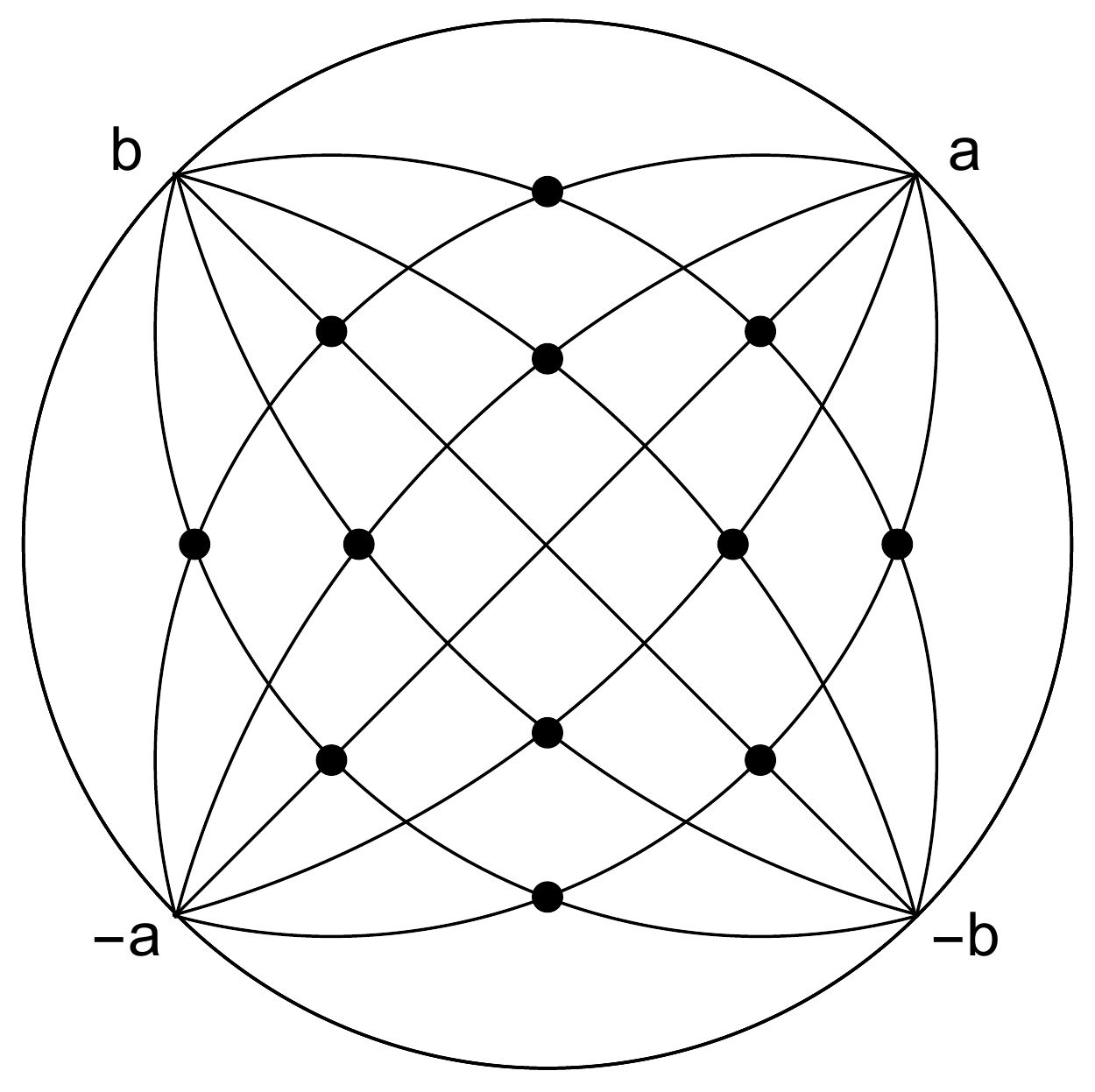}\end{center}
\caption{\label{harmonicquads}Points $u$ in the unit disk for which $\xi,\eta$ have arguments which are multiples of $\pi/n$
(and are not both real). Here $n=3$.}
\end{figure}

These solutions form an interesting set of inscribed quadrilaterals: those which are \emph{harmonic} (that is, are M\"obius images of a square)
and have angles which are multiples of $\pi/n$. Indeed, the complex numbers $1,i\frac{u+a}{u-a},-\frac{(u+a)(u+b)}{(u-a)(u-b)},-i\frac{u+b}{u-b},$
sum to zero and form the sides of such a harmonic quadrilateral.

\section{Limit shapes}

The dimer limit shape theory of Cohn, Kenyon and Propp \cite{CKP} is fairly robust, in the sense that it 
can be extended to more general 
settings where the discrete model in question is described
by a one-dimensional \emph{height function}, for example the $6$-vertex model (although
the strict convexity of the surface tension function, which is important for uniqueness of the limit shape, has not been established for this model).

In the current setting, however, we cannot generally describe a configuration by a height function, 
so we need to extend the limit shape theory to this case.
In the case of zero mass it is possible use Temperley's bijection between
spanning trees and dimers to give the limit shape theory directly (using the height function), 
but this method does not work when there are $k$-pronged singularities for $k$ odd, 
as is the case of Figure \ref{limitshapefig}, for example.
More importantly \emph{no such bijection} is known in the case $M>0$ of the massive laplacian determinant. 
Furthermore there are a number of other
models, for example the Fortuin-Kasteleyn random cluster models, in which we again have a two-parameter family of 
Gibbs measures (defined similarly using infinite parallel components) but again no height function.
The methods discussed here extend the limit shape theory of  \cite{CKP} to apply to these ``banded" models.

\subsection{Measured foliations}

Let $U$ be a simply connected domain in $\C$ with smooth boundary, and $S\subset U$ a finite set which serve as singular points.
A \emph{singular measured foliation} (we say \emph{measured foliation} for short) $\cal F$ on $U$ is an atlas of coordinate
charts for $U$ in $\R^2$ for which the coordinate change maps near a nonsingular point (a point not in $S$)
are diffeomorphisms preserving $dy$ (here $x,y$ are coordinates of $\R^2$). Thus at a nonsingular point $(u,v)\in U$
there is a locally defined notion of ``horizontal coordinate" $y$.
In other language a measured foliation is a decomposition of $U\setminus S$ into disjoint curves, 
where locally the decomposition is a diffeomorphic image of the decomposition of $\R^2$ into horizontal lines, where the diffeomorphism
preserves the spacing between lines.

At a singular point the foliation is locally $k$-pronged, that is, has a singular leaf\footnote{
At the risk of causing confusion with spanning tree terminology we call the components of the foliation \emph{leaves}.} 
consisting of $k\ge3$ rays emanating
from the singularity. Allowing $k$-prong singularities for $k$ odd means that the sign of $``dy"$ changes sign around such a singularity.
A measured foliation can thus be described globally by an unsigned closed $1$-form on $U\setminus S$, 
that is, an equivalence class of (locally defined) $1$-forms under the equivalence $\omega\sim -\omega$.
In local coordinates it is thus an expression $\om = |f(u,v)du+g(u,v)dv|$. Notationally we identify a measured foliation $\cal F$ with its 
unsigned one-form $\om$.

It is convenient to generalize this notion of measured foliation to consider measured foliations supported only on subsets of $U$;
equivalently, we allow $\om$ to not have full support. Our smoothness assumptions guarantee
that the support is the closure of an open set. 

Two measured foliations $\om,|\omega'|$ with the same singular set $S$ are \emph{isotopic} if there is a diffeomorphism of $U$ fixing pointwise 
$\partial U\cup S$, isotopic to the identity fixing the boundary and $S$, and sending $\om$ to $|\omega'|$.

For simplicity we consider in this paper only measured foliations without closed leaves and in which the leaves are transverse to the boundary. (Both these conditions can be relaxed and we get a slightly more general statement of Theorem \ref{limitshapethm} below.)

\subsection{Density of lines}

The Euclidean metric on $U$ allows us to identify cotangent vectors with tangent vectors. 
Given a singular measured foliation $\om=|dy|$ on $U$,
at a nonsingular point $p\in U$ there is a vector $\nabla y$, well defined up to sign,
such that $\nabla y\cdot q = \omega(q)$ for every vector $q$.  This is the \emph{gradient field} of $\om$. 
For a polygon $N\in\R^2$ (symmetric about the origin) we say $\om$ is \emph{$N$-Lipschitz} if $\nabla y\in N$ at every nonsingular point $p\in U$. 

Given a unit tangent vector $q\in T_p(U)$ at a nonsingular point $p\in U$, we define the \emph{density of lines in direction $q$ at $p$}
to be simply $|\omega(q)| = |\nabla y\cdot q|$. 
The condition of being $N$-Lipschitz means that the density of lines of $\om$ is not too great in any direction: small enough that $\om$ 
can be approximated by a grove (see Section \ref{approxsection} below).
The \emph{boundary density} of $\om$ is the density of lines along the boundary, that is, in the direction
of the boundary tangent. The \emph{boundary height function} is a locally defined function whose tangential derivative is the boundary density.

\subsection{Orientation cover}

If $\om$ has any $k$-prong singularities for $k$ odd, there is a two-sheeted branched cover $\tilde U$ of $U$, 
branched exactly over these singularities,
and a lift $\tilde\om$ on $\tilde U$ for which the leaves have a consistent orientation (each singularity of the lifted foliation
is even-pronged).
In this case there is a global coordinate $\tilde y$ (away from the lift of $S$) such that $\tilde\om = |d\tilde y|$.
The $1$-form $d\tilde y$ changes sign under the deck transformation exchanging the sheets of the cover. 
The function $\tilde y$ is called the \emph{global height function}.

\subsection{Approximation}\label{approxsection}

Let $U_\eps\subset \eps\G$ be a connected subgraph of $\eps\G$ approximating $U$ as $\eps\to 0$: we can take $U_\eps = U\cap\eps\G$
with some tweaks near the boundary so that it is connected.

For each singularity $s\in S$ let $s_\eps$ be a face of $U_\eps$ containing $s$.
Let $\partial U_\eps$, the boundary of $U_\eps$, be the set of vertices of $U_\eps$ with a neighbor in $\eps\G$ outside $U_\eps$. 

Recall that a (massive) grove of $U_\eps$ is a spanning forest, every component of which contains either a root (a marked vertex) or contains
at least one boundary vertex
(or both); a component is \emph{special} if it is unrooted.
We consider only groves $F$ in which each special component contains exactly two boundary vertices.  
For each special component $C$ of $F$
there is a unique path between these boundary points, the \emph{trunk} of $C$. 

Given a massive grove $F$ of $U_\eps$ of the above type we associate to it a \emph{discrete measured foliation}
$|dy_\eps|$, which is an unsigned element of the cohomology $H^1(U_\eps,\Z)$, as follows. For a simple path $\gamma$ in the dual graph, 
$$\int_{\gamma}|dy_\eps|=\eps\sum_C |C\wedge\gamma|$$
where the sum is over the set of trunks $C$ of $F$, and
$C\wedge\gamma$ is the algebraic number of crossings of $C$ with $\gamma$. Thus $|dy_\eps|$ is the ``flow" of $F$, 
when we consider the flow to be concentrated on the trunks.
This quantity is invariant under isotopy of $\gamma$ fixing its endpoints, as long as the isotopy does not cross any singularities.

We say that a massive grove $F$ of $U_{\eps}$ is 
\emph{isotopic} to a measured foliation $\om$ of $U$ if the trunk of
each special component of $F$ is isotopic in $U$, fixing the boundary and fixing the singularities, 
to a leaf of $\omega$ whose boundary points are within $O(\eps)$ of those of the corresponding component of $F$.

For a sequence $\eps_j\to0$, let $\{F_j\}_{j=1,2,\dots}$ be a sequence of massive groves of $U_{\eps_j}$.
We say $F_j$ \emph{approximates} 
$\om$ if the corresponding forms $|dy_j| = |dy_{\eps_j}|$ converge weakly to $\om$ in the following sense:
for any fixed smooth path $\gamma$ between two points in $U$ avoiding singularities, 
$$\int_{\gamma} \om = \lim_{j\to\infty}\int_{\gamma}|dy_j|.$$
This says that both the directions and the density of leaves of $F_j$ converge to those of $\om$. 

Note that if $F_j$ approximates $\om$ then almost all (except for a fraction tending to zero of)
its special components are isotopic to leaves of $\om$. 

Finally, if $F$ is a massive grove of $U_\eps$ let $\Omega(F)$ be the set of massive groves $F'$ of $U_\eps$
which connect the same boundary points as $F$ and for which each trunk of $F'$ is isotopic fixing $S$ to the corresponding
trunk of $F$. 

\subsection{Surface tension}

The \emph{surface tension} of a measured foliation $\omega$ is 
\be\label{surfacetension2}\ent(\omega) = \iint_U \sigma(\om)\,du\,dv,\ee
where, writing $\om=|sdv+tdu|,$ 
$\sigma(\om)=\sigma(s,t)=\sigma(-s,-t)$ is minus the free energy of $\mu_{s,t}$, see (\ref{surfacetension1}).

\subsection{Limit shape theorem}

\begin{theorem}\label{limitshapethm}
Let $\G$ be a biperiodic weighted planar graph as above with unit flow polygon $N$.
Fix $M\ge0$. Let $U$ be a piecewise smooth simply connected domain, and $S\subset U$ finite.  Let $\om$ be an $N$-Lipschitz singular measured foliation with singularities $S$, 
with leaves transverse to the boundary and in which every leaf begins and ends on the boundary.
For a sequence $\eps_j>0$ converging to zero as $j\to\infty$, 
let $U_j\subset \eps_j\G$ be a connected subgraph of $\eps_j\G$ approximating $U$.
Let $F_j$ be a massive grove of $U_j$ with $\{F_j\}_{j=1,2,\dots}$ approximating $\omega$.
Let $\mu_j$ be the grove measure on $\Omega(F_j)$.
Then as $j\to\infty$, a $\mu_j$-random grove approximates the measured foliation 
$|\omega_0|$,
where $|\omega_0|$ is the unique measured foliation isotopic to $\om$ and minimizing the 
surface tension $\ent(\xi).$
\end{theorem}

\begin{proof}
The statement and proof are similar to the proof of the limit shape theorem for domino tilings in \cite{CKP},
with some small differences.

Let $\Omega(\om)$ be the space of $N$-Lipschitz measured foliations isotopic to $\om$.
The proof is based on three facts, which we establish in turn.

\begin{enumerate}
\item $\Omega(\om)$ is a compact metric space.
\item There is a unique minimizer $|\omega_0|\in\Omega(\om)$ to the surface tension $\ent$ of (\ref{surfacetension1}).
\item For a sequence $\{F_j\}_{j=1,2,\dots}$ approximating $\om$, the surface tension $\ent(\om)$ is minus the exponential growth rate of the weighted sum of 
configurations in $\Omega(F_j)$. 
\end{enumerate}

Before we prove each of these, let us show how they complete the proof of the theorem.
By compactness, for fixed $\delta>0$, $\Omega(\om)$ can be covered by a finite number of $\delta$-balls $B_\delta(|\omega_k|)$, with $1\le k\le n_\delta$ for some $n_\delta$. 
By the growth condition, for each $j$ and $k$,
$$\ent(|\omega_k|) = \lim_{j\to\infty} \eps_j^2\log Z(\Omega(F_{k,j}))$$
where $\{F_{k,j}\}_{j=1,2,\dots}$ is a sequence approximating $|\omega_k|$. 
By uniqueness of the minimum, the ball which contains $|\omega_0|$ has larger growth rate than any other ball,
so as $j\to\infty$ the measure concentrates on this ball. Thus with probability tending to $1$ a random element of $\Omega(F_j)$
will lie within $\delta$ of $|\omega_0|$.

Now to prove the above facts. 
The Lipschitz condition guarantees that $\Omega(\om)$ is compact: it is a closed subset of the compact space
of unsigned $N$-Lipschitz one-forms. 
By passing to the orientation cover $\tilde U$ of $U$, every measured foliation in $\Omega(\om)$ has a global height function $\tilde y$ 
which is antisymmetric under deck transformation,
and conversely every antisymmetric global height function $\tilde y$ on $\tilde U$ satisfying the condition that $|d\tilde y|$ is $N$-Lipschitz
defines a measured foliation. The set of such height functions with fixed boundary condition and satisfying a Lipschitz condition
is a metric space (in the uniform metric). This metric restricts to a metric on the closed subspace of height functions arising
from a fixed isotopy class.

Now to prove the uniqueness of the optimizer,  fact 2.
Let $|\omega_1|$ be a foliation; one can triangulate $U$ with triangles whose vertices are in $S$ or on a boundary,
do not contain any singular points in their interior or on their edges (only at their vertices)
and whose edges are transverse to $|dy_1|$. 
On each triangle, there is a local coordinate $y_1$ which is monotone on each edge of the triangle. 
The isotopy class of $|dy_1|$ is determined by the integrals of $|dy_1|$ along the edges of the triangles.

Let $|dy_2|$ be a nearby isotopic foliation; it will also be transverse to the triangle edges, and on each triangle
will also have a local coordinate $y_2$ monotone on the edges of the triangle.
We can choose these coordinates $y_1$ and $y_2$
so that $dy_1$ and $dy_2$ have the same sign on each edge. Then for $t\in[0,1]$, $tdy_1+(1-t)dy_2$ is also nonsingular
and has the same sign on each edge. If two triangles share an edge, the form $|t dy_1+(1-t) dy_2|$ is the same for both
triangles, and so these local forms piece together to give an unsigned 1-form on $U$.  This foliation has the same isotopy class as
$|dy_1|$ and $|dy_2|$, since it has the same integrals on edges of the triangles. This gives a ``linear" interpolation between 
$|dy_1|$ and $|dy_2|$; since $\sigma$ is a convex function, $\ent$ is convex on this interpolation.

Now suppose $\omega_1$ and $\omega_2$ are distinct minimizers of $\ent$. On the above interpolation between them, 
and on any triangle where they differ, convexity of $\ent$ implies that
$\ent$ is smaller at any $t\in(0,1)$ than at $t=0$ or $t=1$, a contradicting minimality of both $\omega_1$ and $\omega_2$. 
So there is a unique surface tension minimizer. 

Finally it remains to show that $\ent(\omega)$ is the growth rate of configurations in $\Omega(F_j)$.
This follows the proof in \cite{CKP} exactly, so we simply sketch the argument. 
Fix $\eps>0$ small and triangulate $U$ into small triangles, with sides of order $\Theta(\sqrt{\eps})$ but with angles bounded below, and so that any singularities
only occur at vertices of the triangles. By the Lipschitz condition on $\om$ 
and Rademacher's Theorem, $\om$ is close to constant (that is, $\om=|sdv+tdu|$ for $s,t$ nearly constant functions) on almost all triangles
(all except for a fraction tending to zero of the triangles). A grove close to $\om$ has the property that on almost all triangles,
it lies close to $\om$. On a triangle $T$ where $\om$ is close to a constant $|sdv+tdu|$,
the contribution to the weight of groves close to $\om$ is $e^{A\sigma_{s,t}/\eps^2(1+o(1))}$ where $A$ is the area of the triangle.
The product of these contributions over all triangles gives (upon taking logs and multiplying by $\eps^2$) the surface tension $\ent(\om)$.
\end{proof}

\section{Open question}

The branch of the UST on $\Z^2$ was famously shown to be described by SLE(2) by Lawler, Schramm and Werner in \cite{LSW}.
For a fixed $M>0$, however, 
if one takes a crossing of a square grid strip graph of width $n$ (as in Figure \ref{mt}, panel 2 for width $4$), as $n\to\infty$ the scaling limit of the crossing branch 
(dividing both coordinates by $n$) will be a straight line.  
One can attempt to make interesting limits by sending $M\to 0$ at the same time as the width $n$ goes to $\infty$.
What are the nontrivial scaling limits for this random curve, which interpolate between a straight line and SLE(2)?


\end{document}